\documentclass[12pt,righttag]{amsart}
\usepackage{amsmath,amsthm,amssymb,amscd, amstext, amsopn, amsxtra}

\usepackage{amsfonts}
\usepackage{verbatim}
\usepackage{calc}
\usepackage{amscd}
\usepackage[svgnames]{xcolor}
\usepackage{tikz}
\usepackage{pgf,pgfarrows,pgfnodes,pgfautomata,pgfheaps,pgfshade}
\usepackage[latin1]{inputenc}
\usepackage{multimedia}
\usepackage[T1]{fontenc} 
\usepackage{times}
\usepackage{graphicx}
\pgfdeclarelayer{background layer}
\pgfdeclarelayer{foreground layer}
\pgfsetlayers{background layer,main,foreground layer}

\usepackage{graphics}
\usepackage{epsfig}

\newlength\cellsize 
\savebox2{%
\begin{picture}(18,18)
\put(0,0){\line(1,0){18}}
\put(0,0){\line(0,1){18}}
\put(18,0){\line(0,1){18}}
\put(0,18){\line(1,0){18}}
\end{picture}}
\newcommand\cellify[1]{\def\thearg{#1}\def\nothing{}%
\ifx\thearg\nothing
\vrule width0pt height\cellsize depth0pt\else
\hbox to 0pt{\usebox2\hss}\fi%
\vbox to 18\unitlength{
\vss
\hbox to 18\unitlength{\hss$#1$\hss}
\vss}}
\newcommand\tableau[1]{\vtop{\let\\=\cr
\setlength\baselineskip{-16000pt}
\setlength\lineskiplimit{16000pt}
\setlength\lineskip{0pt}
\halign{&\cellify{##}\cr#1\crcr}}}
\savebox3{%
\begin{picture}(15,15)
\put(0,0){\line(1,0){15}}
\put(0,0){\line(0,1){15}}
\put(15,0){\line(0,1){15}}
\put(0,15){\line(1,0){15}}
\end{picture}}
\newcommand\expath[1]{%
\hbox to 0pt{\usebox3\hss}%
\vbox to 15\unitlength{
\vss
\hbox to 15\unitlength{\hss$#1$\hss}
\vss}}

\numberwithin{equation}{section}
\theoremstyle{definition}
  \newtheorem{theorem}{Theorem}[section]
  
  \newtheorem{proposition}[theorem]{Proposition}
  
  \newtheorem{claim}[theorem]{Claim}
  \newtheorem {corollary}[theorem]{Corollary}
  
  \newtheorem{example}[theorem]{Example}
\theoremstyle{remark}
 \newtheorem{remark}[theorem]{Remark}

\newcommand{\n}{{n}}  
\newcommand{\m}{{m}}  
\newcommand{\W}{{W}}

\newcommand{\Cn}{{C_\n}}
\newcommand{\C}{{\mathcal C}} 
\newcommand{\Afund}{{\mathcal A}_0}
\newcommand{\A}{{\mathcal A}}
\newcommand{\Hak}[2]{{H}_{#1, #2}}
\newcommand{\Hp}[2]{{{\Hak{#1}{#2}}^+}}
\newcommand{\Hpm}[2]{{{\Hak{#1}{#2}}^\pm}}
\newcommand{\Hn}[2]{{{\Hak{#1}{#2}}^-}}

\newcommand{\Hyp}[1]{{\mathcal H}_{#1}}
\newcommand{\roots}{{\Delta}}
\newcommand{\pos}{{\roots^+}}
\newcommand{\negv}{{\roots^-}}
\newcommand{\Inv}{{\mathrm {Inv}}}
\newcommand{\e}{{\varepsilon}}  
\newcommand{\affroots}{{\widetilde \roots}}
\newcommand{\affpos}{{\widetilde \pos}}
\newcommand{\affneg}{{\widetilde \negv}}
\newcommand{\affPi}{{\widetilde \Pi}}

\newcommand{\omitt}[1]{}
 
\newcommand{\wt}{{\mathrm wt}}

\newcommand{\Region}{{\mathfrak A}_\m}

\newcommand{\R}{{\mathbb R}}

\newcommand{\Z}{{\mathbb Z}}

\def\aff{\widehat}

\def\S{{{\mathcal S}}}
\def\Sn{{\S_\n}}
\def\affS{\aff \Sn}

\def\m{{m}}

\newcommand{\brac}[2]{{\langle #1\mid#2\rangle}}

\def\mod{{\rm mod\,}}

\def\*{\mathop{\otimes}}
 
\begin{document}
\author[S.~Fishel]{Susanna Fishel}
\author[M.~Vazirani]{Monica Vazirani}
\address[M.~Vazirani]{UC Davis\\
Department of Mathematics\\
One Shields Ave\\
Davis, CA 95616-8633 }
\address[S.~Fishel]{
Department of Mathematics and Statistics\\
Arizona State University\\
P.O. Box 871804\\
Tempe, AZ 85287-1804
}
\thanks{Both authors wish to
thank AIM and the SQuaREs program where this work was started.}
\title[Bijection from Shi regions to
cores]{A bijection between dominant Shi regions and core partitions}

\date{\today}
\begin{abstract}
It is well-known that Catalan numbers $C_\n = \frac{1}{\n+1} \binom{2\n}{\n}$
count the number of dominant regions in the Shi arrangement of type $A$,
and that
they also count partitions which are both $\n$-cores as well as $(\n+1)$-cores.
These concepts have natural extensions, which we call
here  the $\m$-Catalan numbers and $\m$-Shi arrangement.
In this paper, we construct a bijection between dominant
regions of the $\m$-Shi arrangement and partitions which are
both  $\n$-cores as well as $(\m\n+1)$-cores.
The bijection is natural in the sense that it commutes with the
action of the affine symmetric group.
\end{abstract}
\maketitle


\section{Introduction }
\label{sec-intro}

In this paper, we build on the work of Anderson \cite{Anderson} to
give a direct bijection between dominant regions in the ``extended''
$\m$-Shi arrangement
of type $A_{\n-1}$ and partitions that are simultaneously $\n$-cores
and $(\m\n+1)$-cores.

Anderson's result, that
$\frac{1}{s+t} \binom {s+t}{t}$
counts
the number of partitions which are
both $t$-cores and $s$-cores 
generalizes the
well-known interpretation of Catalan numbers $\Cn$ as counting
$\n$-cores that are also $(\n+1)$-cores.
Catalan numbers are also known to count dominant regions in the
Shi arrangement, termed here the  $1$-Shi arrangement. 
(See \cite{Stanley-ec, Stanley-catalan}.)
Our result provides a direct bijective proof between 
two more general sets of  combinatorial objects
counted by the higher Catalan numbers, called here the
$\m$-Catalan numbers.

Our bijection is $\W$-equivariant in the following sense.
In each connected component of the $\m$-Shi hyperplane arrangement 
of type $A_{\n-1}$, there is exactly one ``representative,''  or $\m$-minimal,
alcove closest to the fundamental
alcove $\Afund$. 
Since
 the affine Weyl group
$\W$ acts freely and transitively on the set of alcoves,
there is a natural way to associate an element $w \in \W = \affS$
to any alcove $w^{-1} \Afund$, and to
 this one in particular. 
There is also a natural action of $\affS$ on partitions, whereby the
orbit of the empty partition $\emptyset$ is precisely the $\n$-cores. 
We will show that $w \emptyset$ is also an $(\m\n+1)$-core
and that all such $(\m\n + 1)$-cores that are also $\n$-cores
can be obtained this way. 

Roughly speaking, to each $\n$-core  $\lambda$ we can associate
an integer  vector $\vec n (\lambda)$ whose entries sum to zero.
When $\lambda$ is also an $(\m\n + 1)$-core, these entries satisfy certain
inequalities. 
On the other hand, these are precisely the inequalities that
describe when a dominant alcove is $\m$-minimal.

 Similar techniques can
be used to show the $\n$-cores which are also $(\m\n-1)$-cores are in bijection
with the {\it bounded\/} dominant regions in the $\m$-Shi arrangement.  This observation,
which we develop in \cite{FV},  was pointed out by Nathan Reading  whom we wish to thank.

 The article is organized as follows. In Section \ref{sec-Shi} we review
facts about Coxeter groups and root systems of type $A$.
Sections \ref{sec-inversion}
and \ref{sec-minimal} 
explain how  the position of $w^{-1} \Afund$ relative to our system of
affine hyperplanes is captured by the action of $w$ on affine roots
and that $\m$-minimality can be  expressed by certain inequalities on
the entries of $w(0,0,\ldots,0)$. 
In Section \ref{sec-abacus}  we review facts about
core partitions and  in particular remind the reader how to
associate an element of the root lattice to each core.
Our main theorem,
the bijection between dominant regions of the $\m$-Shi arrangement and special
cores, is in Section \ref{sec-bij}.
Section \ref{sec-inverse}  describes a related bijection on alcoves.
In Section \ref{sec-Narayana},
we derive further results refining our bijection between alcoves and
cores that involve Narayana  numbers.   Here we also observe that  an $\n$-core
is also an $(\m\n+1)$-core when it has $\le m$ removable boxes of any fixed residue.


\section{The type $A$ root system, Shi arrangement, and Weyl group}
\label{sec-Shi}

\subsection{Notation}
Let $\{\e_1, \ldots, \e_\n \}$ be the standard basis of $\R^n$ and $\brac{\,}{\,}$
be the  bilinear form for which this is an orthonormal basis.
Let $\alpha_i = \e_i - \e_{i+1}$. Then $\Pi = \{\alpha_1, \ldots, \alpha_{\n-1} \}$
is a basis of $V =  \{ (a_1, \ldots, a_\n ) \in \R^n \mid \sum_{i=1}^n a_i =0 \}$.
We let $Q =  Q^{(\n-1)} = \bigoplus_{i=1}^{n-1} \Z \alpha_i$ be identified 
with the root lattice of type $A_{\n-1}$.
The elements of $\roots = \{\e_i - \e_j \mid i \neq j\}$ are called roots
and we say a root $\alpha$ is positive, written $\alpha > 0$, if $\alpha \in
\pos = \{\e_i - \e_j \mid i < j\}$. We let $\negv = - \pos$ and say $\alpha < 0$
if $\alpha \in \negv$. Then $\Pi$ is the set of simple roots. 

We define a system of affine hyperplanes
$$\Hak{\alpha}{k}  = \{ v \in V \mid \brac{v }{\alpha} = k \}$$
given $\alpha \in \roots, k \in \Z$. Note $\Hak{-\alpha}{-k}  =\Hak{\alpha}{k}  $
so we usually take $k \in \Z_{\ge 0}$. 

The  extended Shi arrangement $\mathcal{S}^\m_\n$,
here called the  {\it $\m$-Shi arrangement},
is 
the  collection of hyperplanes $\Hyp{\m} = 
\{ \Hak{\alpha}{k} \mid \alpha \in \pos, -m < k \le m \}$.
This
   arrangement can be defined for all types; here we are concerned
   with type $A$.

We denote the (closed) half spaces
$\Hp{\alpha}{k} = \{ v \in V \mid \brac{v}{\alpha} \ge k \}$ and 
$\Hn{\alpha}{k} = \{ v \in V \mid \brac{v}{\alpha} \le k \}$.
Then the {\it dominant} chamber of $V$ is 
$\bigcap_{i=1}^{n-1} \Hp{\alpha_i}{0}$ (also referred to as the
fundamental chamber in the literature).
In this paper, we are primarily concerned with the connected components of
the hyperplane arrangement complement
$V \setminus \bigcup_{H \in \Hyp{\m}} H$
and in particular the connected
components in the dominant chamber.
For ease of notation we will refer to these as (dominant) regions of
the $\m$-Shi arrangement.

The extended Shi arrangement was defined by Stanley in
\cite{S1998}. The arrangement $\mathcal{S}^1_\n=\mathcal{S}_\n$ is known
as the Shi arrangement and was first considered by Shi \cite{Shi1986,
  Shi1987} and later by Headley \cite{H1994,H1994b}.  The extended Shi
arrangement was further studied in, for example,
\cite{AT2006,AtL1999,A2004,A2005a,PS2000}. In \cite{A2005a},
Athanasiades gave  formulas for the number of dominant
regions of the $\m$-Shi arrangement  $\mathcal{S}^\m_\n$ in
any type.

Each connected component of
$V \setminus \bigcup_{\stackrel{\alpha \in \pos}{k \in \Z}} \Hak{\alpha}{k}$
is called an alcove and the {\it fundamental alcove} is
$\Afund =$ the interior of $ \Hn{\theta}{1}  \cap \bigcap_{i=1}^{n-1} \Hp{\alpha_i}{0}$, 
where $\theta = \alpha_1 + \cdots + \alpha_{\n-1} = \e_1 - \e_\n$.

\subsection{The affine symmetric group}
The affine symmetric group $\affS$ acts on $V$ (preserving $Q$)
via affine linear transformations,
 and acts freely and transitively on the set of alcoves. 
We thus identify each alcove $\A$ with the unique $w \in \affS$ such that 
$\A = w \Afund$. 
\begin{align*}
\affS = \langle s_1, \ldots, s_{\n-1}, s_0 \mid
s_i^2 = 1, \quad  & s_i s_j = s_j s_i \text{ if } i \not\equiv j \pm 1 \mod n, \\
& 
 s_i s_j  s_i = s_j s_i s_j \text{ if } i \equiv j \pm 1 \mod n \rangle
\end{align*}
for $n > 2$, but ${\aff \S_2} = \langle s_1, s_0  \mid s_i^2 = 1 \rangle$. 
Each simple generator $s_i$ acts by reflection
with respect to the simple root $\alpha_i$.
(The simple root $\alpha_0$ is discussed below; note $s_0$ acts as reflection
over an affine hyperplane of $V$.)
More specifically, 
the action is given by 
\begin{gather*}
s_i   (a_1, \ldots, a_i, a_{i+1}, \ldots, a_\n) 
= (a_1, \ldots,  a_{i+1}, a_i,\ldots, a_\n) \quad
\text{ for $i \neq 0$, and} \\
s_0   (a_1, \ldots,  a_\n) 
= (a_\n +1, a_2 , \ldots,  a_{\n-1},  a_1 - 1).
\end{gather*}
Note $\Sn$ preserves $\brac{\;}{\;}$, but $\affS$ does not.

We note that $\affS$ contains a normal subgroup consisting of translations
by elements of $Q$, $\{t_\gamma \mid \gamma \in Q\}$, 
and that $\affS = \Sn \ltimes Q$. 
In terms of the coordinates above, if
$\gamma = (\gamma_1, \ldots, \gamma_\n) \in Q \subseteq V$
then $t_\gamma(a_1, \ldots, a_\n) = (a_1 + \gamma_1, \ldots, a_\n + \gamma_\n)$.
Consequently, we may express any $w \in \affS$ as $w = u t_\gamma$
for unique $u \in \Sn$, $\gamma \in Q$, or equivalently
$w = t_{\gamma'} u$ where $\gamma' = u(\gamma)$. Note that $\gamma' =
w(0,\ldots, 0)$. 

The finite root system above extends to an affine root system
$\affroots = \{ k\delta + \alpha \mid k \in \Z, \alpha \in \roots \}
= \affpos \cup \affneg$ where
$\affpos =  \{ k\delta + \alpha \mid k \in \Z_{\ge 0}, \alpha \in \pos \}
  \cup \{ k\delta + \alpha \mid k \in \Z_{> 0}, \alpha \in \negv \}$,
$\affneg = - \affpos$, and the simple roots are now $\affPi = \Pi \cup \{ \alpha_0\}$.
Here $\delta$ is the imaginary root,
and $\alpha_0 = \delta - \theta$.
Again we write $\alpha > 0$ if $\alpha \in \affpos$, 
$\alpha < 0$ if $\alpha \in \affneg$. 

We can picture $V$ sitting in the span of $\affPi$ as an affine subspace;
whereas $\affS$ acts on the larger space linearly, it acts on $V$ by affine
linear transformations. 
 In most of this paper, we found it more useful
to work with $V$, its coordinate system, and affine hyperplanes.
However, we could have chosen to express everything in terms of
the span of $\affPi$, as we found useful to do in Section \ref{sec-inversion}. 
In terms of affine roots,
the action of the simple reflection  $s_i \in \affS$, $0 \le i < n$, is given by
\begin{gather}
\label{eqn;rootsi}
s_i (\gamma) = \gamma - \brac{\gamma}{ \alpha_i}\alpha_i,
\end{gather}
where we have extended $\brac{\,}{\,}$ as appropriate (given by the
Cartan matrix of type $A_{n-1}^{(1)}$).
We can also re-express the 
 action of translations as
\begin{gather}
\label{eqn;roottranslation}
t_\gamma (\alpha) = \alpha - \brac{\gamma}{ \alpha}\delta 
\end{gather}
for $ \alpha \in \affroots$, $\gamma$ in the integer span of $\affPi$ or indeed
$\gamma$ in the affine weight lattice.
Note, in the case $\gamma \in Q$, $\alpha \in \roots$, $u \in \Sn$,
expressions like $\brac{\gamma}{ u(\alpha)}$ are unambiguous as they agree in
either setting.

Notice that the length $\ell(w) = | \{\alpha > 0 \mid w(\alpha) < 0 \} |$
for $w \in \affS$ is just the minimal number of affine
hyperplanes separating $w \Afund$
from $\Afund$. 


\section{Inversion sets}
\label{sec-inversion}

Define
$\Inv(w) =
\{\alpha > 0 \mid w(\alpha) < 0 \}$, for $w \in \affS$.
While we could have defined $\Inv(w)$ in terms of the $\affS$ action on
$V$, it is more conventional to define it in term of the action on $\affroots$.
It is well known that $\ell(w) = | \Inv(w) |$ and that this
number also counts the number of affine hyperplanes 
$\Hak{\alpha}{k}$ separating $\Afund$ from $w^{-1} \Afund$.
This section reminds the reader of the exact correspondence.

In this section
it is more convenient to express the $\affS$ action in terms of affine roots
than in terms of the coordinates of $V$,
because of how $\Inv(w)$ is defined. 

\begin{proposition}
\label{prop;inv}
Let $w \in \affS$ and  $\alpha + k \delta \in \affpos$, i.e. $k > 0, \alpha \in \roots$
or $k=0, \alpha \in \pos$.
Then
$\alpha + k \delta \in \Inv(w)$
iff
$w^{-1} \Afund \subseteq \Hp{-\alpha}{k}$
\end{proposition}
\begin{proof}
We first write $w = t_\gamma u$ for $\gamma \in Q$, $u \in \Sn$.
Note, in terms of affine roots, 
 $w(\alpha + k \delta) =  u(\alpha) + (k-\brac{\gamma}{u(\alpha)}) \delta$.
Assume $\alpha + k \delta \in \Inv(w)$ which 
 implies either $k-\brac{\gamma}{u(\alpha)} = 0,
u(\alpha) \in \negv$ or  $k-\brac{\gamma}{u(\alpha)} < 0$.

On the other hand, consider $v \in w^{-1} \Afund$, so that $w(v) \in \Afund$.
Then $$0\le \brac{w(v)}{\eta} \le 1 \quad \text{for all $\eta \in \pos$}$$
and so translating by $-\gamma$ we get
 $-\brac{\gamma}{\eta} \le \brac{u(v)}{\eta} \le 1-\brac{\gamma}{\eta}$.
First suppose $u(\alpha) \in \negv$ and let $\eta = - u(\alpha)$ for $k, \alpha$ as
above. 
Then in particular 
\begin{gather*}
k \le \brac{\gamma}{u(\alpha)}= -\brac{\gamma}{u(-\alpha)}
\le \brac{u(v)}{u(-\alpha)} = \brac{v}{-\alpha} 
\end{gather*}
so $v \in \Hp{-\alpha}{k}$.

If instead $u(\alpha) \in \pos$, we must have $k-\brac{\gamma}{u(\alpha)} \le -1$,
so taking $\eta = u(\alpha)$
\begin{gather*}
k \le -1+ \brac{\gamma}{u(\alpha)}
\le -\brac{u(v)}{u(\alpha)} = \brac{v}{-\alpha} 
\end{gather*}
so  again $v \in \Hp{-\alpha}{k}$.

The converse is straightforward; similar methods show that if
$w(\alpha + k \delta)  \in \affpos$ then 
$w^{-1} \Afund \subseteq \Hn{-\alpha}{k}$.

\end{proof}

\begin{figure}[!ht] 

\begin{tikzpicture}
\tikzstyle{every node}=[font=\footnotesize]
[4,1,2]

\path [draw = black, very thick, draw opacity = 1, shift = {(0,0)}]
(0,0) -- +(60:6) node[black,above] {$H_{\alpha_1,0}$};

\path [draw = black,very thick, draw opacity = 1] (0,0) -- +(6,0) node[black,right] {$H_{\alpha_2,0}$};

\path [draw = red,very thick,draw opacity = 1] (0,0) -- (0,0);

\path [draw = blue, very thick, draw opacity = 1, shift = {(1,0)}]
(0,0) -- +(60:6) node[black,above] {$H_{\alpha_1,1}$};

\path [draw = green,very thick, draw opacity = 1] (0.5,0.866025) -- +(6,0);
\draw [shift = {(0.5,0.866025)}] (6.5,0) node {$H_{\alpha_2,1}$};

\path [draw = red,very thick,draw opacity = 1] (0.5,0.866025) --
(1,0) node[black,below] {$H_{\theta,1}$};

\path [draw = blue, very thick, draw opacity = 1, shift = {(2,0)}]
(0,0) -- +(60:6) node[black,above] {$H_{\alpha_1,2}$};

\path [draw = gray, draw opacity = 1] (1,1.73205) -- +(6,0);
\draw [shift = {(1,1.73205)}] (6.5,0) node {$H_{\alpha_2,2}$};

\path [draw = red,very thick,draw opacity = 1] (1,1.73205) -- (2,0) node[black,below] {$H_{\theta,2}$};

\path [draw = gray, draw opacity = 1, shift = {(3,0)}]
(0,0) -- +(60:6) node[black,above] {$H_{\alpha_1,3}$};

\path [draw = gray, draw opacity = 1] (1.5,2.59808) -- +(6,0);

\path [draw = red,very thick,draw opacity = 1] (1.5,2.59808) -- (3,0) node[black,below] {$H_{\theta,3}$};

\path [draw = gray, , draw opacity = 1, shift = {(4,0)}]
(0,0) -- +(60:6);

\path [draw = gray,draw opacity = 1]  (2,3.4641) --
+(6,0);

\path [draw = red,very thick,draw opacity = 1] (2,3.4641) -- (4,0)
node[black,below, text = black, draw opacity = 1] {$H_{\theta,4}$};

\fill[fill=yellow, fill opacity=1.](0,0) circle (0.09cm);
\draw[black, fill opacity=1.](0,0) circle (0.11cm);

\fill [draw = white, draw opacity = 0, thick,fill = yellow,
shift={(3.5,0.866025)}](0,0) -- (60:1cm) -- (120:1cm) --(0,0)
node[green,fill opacity=1] at (0cm,0.8cm) {} node[black,fill
opacity=1, draw opacity = 1, inner sep = 2pt] at (0cm,0.6cm)
{${w^{-1}A_0}$} node[green,fill opacity=1] at (0cm,0.4cm) {};

\end{tikzpicture}
  \caption{\small 
When $w^{-1} \Afund \subseteq \Hp{\alpha}{k}$ for $\alpha \in \pos$
}
  \label{fig;hyper_inv}
\end{figure}
%
\begin{example}
\label{ex;hyp_inv_one}
\omitt{ the $3$-core $\lambda = (5,3,2,2,1,1) = w \emptyset$ for}
Consider
$w = s_1 s_2 s_0 s_1 s_2 s_1 s_0 $,
and let $\A = w^{-1} \Afund$, which is pictured in Figure \ref{fig;hyper_inv}.
Observe $w =s_2 t_{(-2,0,2)}$ and so $w^{-1} = t_{(2,0,-2)} s_2$.
Then
$$\A \subseteq {\color{Blue} \Hn{\alpha_1}{3}}
 \cap {\color{Red} \Hp{\theta}{4} }
\cap {\color{LimeGreen} \Hn{\alpha_2}{2}}.$$
Note $w^{-1}(\alpha_0) = {\color{Blue} -\alpha_1 + 3 \delta}$,
$w^{-1}(\alpha_1) = {\color{Red}  \theta - 4 \delta}$,
$w^{-1}(\alpha_2) = {\color{LimeGreen} -\alpha_2+ 2 \delta}$.

Furthermore $\A \subseteq {\color{Blue} \Hp{\alpha_1}{2}}$ and hence also
$\A \subseteq {\color{Blue} \Hp{\alpha_1}{1}}$.
Likewise $\A \subseteq {\color{Red} \Hp{\theta}{k}}$ for
$1\le k \le 4$, and so on.

We observe this is captured by the following:
\begin{align*}
\Inv(w) = \{
&{\color{Blue}  -\alpha_1 + \delta, -\alpha_1 + 2\delta},
\\
&{\color{Red}  
 -\theta + \delta, -\theta + 2\delta, -\theta + 3\delta, -\theta + 4\delta},
\\
& {\color{LimeGreen} -\alpha_2 + \delta }  \}.
\end{align*}
\end{example}

\begin{corollary}
Suppose $w $ 
is a minimal length left coset representative
for
$ \affS/\Sn$.
Then
$\Inv(w)$ consists only of roots of the form
$-\alpha + k \delta$, $k \in \Z_{>0}, \alpha \in \pos$.
Further, if $-\alpha + k \delta \in \Inv(w)$ and $k>1$ then $-\alpha + (k-1) \delta
\in \Inv(w)$.
\end{corollary}
\begin{proof}
Since $w^{-1}$ is a minimal length right coset representative,
 $w^{-1} \Afund$ is in the dominant chamber, and so in $\Hp{\alpha}{0}$ for 
any $\alpha \in \pos$, and in particular never in $\Hp{-\alpha}{k}$ for $k \ge 0$. 

For the second statement, note $\Hp{\alpha}{k} \subseteq \Hp{\alpha}{k-1}$ if $k>0$.
\end{proof}
 
We also remind the reader that when $w^{-1}$ is a minimal length right
coset representative for $ \affS/\Sn$, then we may write
 $w^{-1} = t_{\gamma'}u$ where
$u \in \Sn$ and 
$\gamma'$ is in the dominant chamber.

\section{$\m$-minimal alcoves}
\label{sec-minimal}
We can identify each connected component of the complement of the $\m$-Shi arrangement 
with the unique alcove $w \Afund$ contained in it such that $\ell(w)$ is smallest.
In this situation we will say the alcove $w \Afund $ is  $\m$-{\em minimal.}
Such alcoves are termed ``representative alcoves'' by Athanasiades.

The following proposition is useful.
For a given alcove, it characterizes the
affine hyperplanes containing its walls and which simple reflections
flip it over those walls (by the right action). 
It can be found in \cite{Shi} in slightly different notation.

\begin{proposition}
\label{prop;hyper}
Suppose $w \Afund \subseteq \Hp{\alpha}{k}$
but 
$w s_i \Afund \subseteq \Hn{\alpha}{k}$
\begin{enumerate}
\item
Then $w(\alpha_i) = \alpha - k \delta$.
\item 
Let $\beta =  w^{-1}  (0, \ldots , 0) \in V$. Then $\brac{\beta}{\alpha_i} = -k$.
\end{enumerate}
\end{proposition} 

\begin{proof}
We first note $\Hak{\alpha_i}{0}$ is the unique hyperplane such that
$\Afund \subseteq \Hp{\alpha_i}{0},$
$ s_i \Afund \subseteq \Hn{\alpha_i}{0}.$
In the case $i=0$ we rewrite this condition as 
$\Afund \subseteq \Hp{-\theta}{1}, 
s_0 \Afund \subseteq \Hn{-\theta}{1}.$

Because $w$ is an isometry, $w(\Hak{\alpha_i}{0})$ must be the unique
hyperplane separating $w \Afund$ from $w s_i \Afund$, and by hypothesis,
this hyperplane is $\Hak{\alpha}{k}$.
Then $w (\Hpm{\alpha_i}{0}) = \Hpm{\alpha}{k}$ implies $w (\alpha_i) = \alpha - k \delta$.

For the second statement, note that we can uniquely write $w = t_\gamma u$ with
$u \in \Sn, \gamma \in Q$ as $\affS = Q \ltimes \Sn$.
Then $\beta = w^{-1}   (0, \ldots, 0)  = u^{-1}(-\gamma )$ and
$\brac{u^{-1}(-\gamma)}{\alpha_i} = \brac{-\gamma}{u (\alpha_i)} =
-\brac{\gamma}{\alpha} = -k$ since
$w(\alpha_i) = t_\gamma(u(\alpha_i)) = u(\alpha_i) - \brac{\gamma}{u (\alpha_i)}\delta
= \alpha - k \delta$. 
\end{proof}

Using the coordinates of $V$, we note $k = \gamma_{u(i)} - \gamma_{u(i+1)}$.

\begin{remark}
\label{remark;min}
Note, if $w \Afund$ is  
$\m$-minimal, 
then whenever 
$k \in \Z_{\ge 0}$ and  $w \Afund \subseteq \Hp{\alpha}{k}$
but
$w s_i \Afund \subseteq \Hn{\alpha}{k}$
then we must have $k \le \m$ in the case $\alpha > 0$ and  $k \le m-1$ in the case
$\alpha < 0$.
\end{remark}
It is easy to see that the condition in Remark \ref{remark;min}
 is not only necessary
but sufficient to describe when $w \Afund$ is $\m$-minimal.
Together with Proposition \ref{prop;inv},  Proposition \ref{prop;hyper}
says
that when $\alpha_i \in \Inv(w)$, $w(\alpha_i) = \alpha - k \delta$
then $k \le \m$,
and for $\beta =  w^{-1}   (0, \ldots , 0)$ that $\brac{\beta}{\alpha_i} \ge
-\m$.

Applying
Remark \ref{remark;min}
to just positive $\alpha$ and alcoves in the dominant
chamber, we get the following corollary. 

\begin{corollary}
\label{cor;minimal}
Suppose $w \Afund$ is in the dominant chamber and $\m$-minimal. 
\begin{enumerate}
\item
If $w \Afund \subseteq \Hp{\alpha}{k}$
but
$w s_i \Afund \subseteq \Hn{\alpha}{k}$
for some $\alpha \in \pos, k \in \Z_{\ge 0}$, then $k \le \m$.
\item
Let $\beta =  w^{-1}   (0, \ldots , 0)$. Then $\brac{\beta}{\alpha_i} \ge -\m$,
for all $i$,
and in particular $\brac{\beta}{\theta} \le \m+1$.
\end{enumerate}
\end{corollary}
\begin{proof}
The first statement follows directly from Proposition \ref{prop;hyper}
 and Remark \ref{remark;min}.
To conclude that the second statement holds for all $i$, note that if $k \le 0$
then automatically $k \le m$.
\end{proof}

\section{Core partitions and their  abacus diagrams}
\label{sec-abacus}

Here we review some well-known facts about $\n$-cores and review
the useful tool of the abacus construction.  Details can be
found in \cite{JamesKerber}.


There is a well-known bijection $\C : \{ \n \text{-cores} \} \to Q$  
that commutes with the action of $\affS$. 
One can use the $\affS$-action to define the bijection, or
describe it directly from the combinatorics of partitions
via the work of Garvan-Kim-Stanton's $\vec n$-vectors \cite{GKS;cranks}
or as described in terms of balanced abaci as in \cite{BJV}. 

Here, we will recall the description from \cite{BJV} as well as
remind the reader of the $\affS$-action on $\n$-cores.

We identify a partition $\lambda = (\lambda_1, \ldots, \lambda_r)$ with
its Young diagram, the array of boxes with coordinates $\{(i,j) \mid
1 \le j \le \lambda_i\}$.  We say the box $(i,j) \in \lambda$ has
{\it residue\/} $j-i \, \mod n$, and in that case, we often refer to it as
a $(j-i \,  \mod n)$-box.
Its {\it hook length\/} $h_{(i,j)}^\lambda$  is   $1 +$
the number of boxes to the right of and below $(i,j)$.

An {\it $n$-core\/} is a partition $\lambda$ such that $n \nmid h_{(i,j)}^\lambda$
for all $(i,j) \in \lambda$. 

We say a box is {\it removable\/} from $\lambda$ if its removal results in 
a partition.  Equivalently its hook length is $1$. 
A box not in $\lambda$ is {\it addable\/} if its union with $\lambda$ results
in  a partition. 

\begin{claim}
\label{claim;add-remove}
Let $\lambda$ be an $n$-core. Suppose $\lambda$ has a removable $i$-box. Then it has no
addable $i$-boxes.  Likewise, if $\lambda$ has an addable $i$-box it has no removable
$i$-boxes. 
\end{claim}
\begin{proof}
If $\lambda$ had both a removable $i$-box $(x,y)$ and an addable $i$-box $(X, Y)$,
then $\lambda$ also contains exactly one of $(x, Y)$ or $(X,y)$ and this box
has hook length $|X-x + y-Y|$ which is 
divisible by $n$, as $y-x \equiv Y-X \equiv i \, \mod n$. 
\end{proof}

$\affS$ acts transitively
on the set of $n$-cores as follows.  Let $\lambda$ be an $n$-core.
Then 
$$ s_i \lambda = \begin{cases}
\lambda \cup \text{all addable $i$-boxes} & \text{$\exists$ any addable $i$-box} \\
\lambda \setminus \text{all removable $i$-boxes} & \text{$\exists$ any removable $i$-box} \\
\lambda & \text{else.}
\end{cases}$$
It is easy to check $s_i \lambda$ is an $n$-core.

In fact $\affS$ acts on the set of all partitions, but this action is
slightly more complicated to describe
(see \cite{MisraMiwa} and Section 11 of \cite{Kashiwara}),
 and involves the combinatorics
of Kleshchev's ``good'' boxes. 
For those readers familiar with the realization of the basic crystal
$B(\Lambda_0)$ of ${\widehat {\mathfrak sl}_n}$ as having nodes 
parameterized by $n$-regular partitions,
$$ s_i \lambda = \begin{cases}
\tilde f_i^{\langle h_i, {\mathrm wt}(\lambda) \rangle}(\lambda)
& \langle h_i, {\mathrm wt}(\lambda) \rangle \ge 0 \\
\tilde e_i^{-\langle h_i, {\mathrm wt}(\lambda) \rangle}(\lambda)
& \langle h_i, {\mathrm wt}(\lambda) \rangle \le 0,
\end{cases}$$
where 
\begin{gather}
\label{eq-wtlambda}
{\mathrm wt}(\lambda) =
\Lambda_0 -  \sum_{(x,y) \in \lambda} \alpha_{y-x \, \mod n},
\end{gather}
and $h_i$ is the co-root corresponding to $\alpha_i$.  (Since we need not make
a distinction between roots and co-roots in type $A$, we could have simply 
substituted $\alpha_i$ for $h_i$ in the expressions above.)
Then the $n$-cores are exactly the $\affS$-orbit on the highest weight node,
which is the empty partition $\emptyset$. 

\subsection{Abacus diagrams}
We can associate to each partition $\lambda$ its abacus diagram.  
When $\lambda$ is an $n$-core, its abacus has a particularly nice form,
and then can be used to construct an element of $Q$. 
This gives us a bijection $\{ n\text{-cores} \} \to Q$  which commutes
with the action of $\affS$. 
We follow \cite{BJV} in describing this bijection,
which rests on the work of \cite{JamesKerber}, 
and we note
this bijection agrees with the $\vec n$-vector construction
of Garvan-Kim-Stanton \cite{GKS;cranks},

Each partition $\lambda = (\lambda_1, \dots, \lambda_r)$ is determined by its
hook lengths in the first column,  the $\beta_k = h_{(k,1)}^{\lambda}$.

An \textit{abacus diagram} is a diagram, with entries from $\Z$ arranged
in 
$\n$ columns labeled $0, 1, \dots, \n-1$, called \textit{runners}. 
The horizontal cross-sections or rows will be called \em levels \em and runner
$k$ contains  the entry labeled by $r \n + k$ on level $r$ where $-\infty < r < \infty$.
 We draw the abacus so that each runner is vertical, oriented with
$-\infty$ at the top and $\infty$ at the bottom, and we always 
put runner $0$ in the leftmost position,
increasing to runner $\n-1$ in the rightmost position.
 Entries in the abacus diagram may be circled; such circled elements are called
\textit{beads}. Entries which are not circled will be called \textit{gaps}. 
We shall say two abaci are equivalent if they differ by adding a constant
to all entries.  (Note, in this case we must cyclically permute the runners
so that runner $0$ is leftmost.)
See Example \ref{ex-abacus} below.

Given a partition $\lambda$  its abacus is any abacus diagram equivalent to
the one obtained by placing beads at 
entries $\beta_k = h_{(k,1)}^{\lambda}$ and all $j \in \Z_{< 0}$.

\begin{remark}
It is well-known that $\lambda$ is an $\n$-core if and only if its abacus is
\textit{flush}, that is to say whenever
there is a bead at entry $j$ there is also a bead at $j - \n$.
\end{remark}

We define the \textit{balance number} of an abacus to be the sum over all
runners of the largest level in that runner which contains a bead.
We say that an abacus is \textit{balanced} if its balance number is zero.
Note that there is a unique abacus which represents a given $\n$-core
$\lambda$ for each balance number.
In particular, there is a unique abacus of $\lambda$ with balance number $0$. 
The balance number for a set of $\beta$-numbers of $\lambda$ will increase by exactly
$1$ when we increase each $\beta$-number by $1$.
On the abacus picture, this corresponds to shifting all of the beads forward
one entry.  (Equivalently, we could add $1$ to each entry, leaving the beads
in place, after which we cyclically permute the runners so that runner $0$ is
leftmost. )

\begin{example}
\label{ex-abacus}
Both abaci below  represent the 4-core $\lambda =(5,2,1,1,1)$. 
The first one is balanced, but the second has balance number $1$.
The boxes of $\lambda$
have been filled with their hooklengths.

\begin{figure}[ht]

\begin{tikzpicture}[fill opacity=1,every to/.style={draw,dotted}]
\tikzstyle{every node}=[font=\footnotesize]
\path (0,0) to (0,4);
\path[shift = {(.75,0)}] (0,0) to (0,4);
\path[shift = {(1.5,0)}] (0,0) to (0,4);
\path[shift = {(2.25,0)}] (0,0) to (0,4);
\shade[shading=ball, ball color = Gainsboro] (0,1) circle (.25) node {8};
\shade[shading=ball, ball color = Gainsboro] (0,1.5) circle (.25) node
{4};
\shade[shading=ball, ball color = Gainsboro] (0,2) circle (.25) node {0};
\shade[shading=ball, ball color = Gainsboro] (0,2.5) circle (.25) node {-4};
\shade[shading=ball, ball color = Gainsboro] (0,3) circle (.25) node {-8};
\fill[fill=white] (.75,1) circle (.25) node {9};
\fill[fill=white] (.75,1.5) circle (.25) node
{5};
\shade[shading=ball, ball color = Gainsboro] (.75,2) circle (.25) node {1};
\shade[shading=ball, ball color = Gainsboro] (.75,2.5) circle (.25) node {-3};
\shade[shading=ball, ball color = Gainsboro] (.75,3) circle (.25) node {-7};

\fill[fill=white] (1.5,1) circle (.25) node {10};
\fill[fill=white] (1.5,1.5) circle (.25) node {6};
\shade[shading=ball, ball color = Gainsboro] (1.5,2) circle (.25) node {2};
\shade[shading=ball, ball color = Gainsboro] (1.5,2.5) circle (.25) node {-2};
\shade[shading=ball, ball color = Gainsboro] (1.5,3) circle (.25) node {-6};

\fill[fill=white] (2.25,1) circle (.25) node {11};
\fill[fill=white] (2.25,1.5) circle (.25) node
{7};
\fill[fill=white] (2.25,2) circle (.25) node {3};
\fill[fill=white] (2.25,2.5) circle (.25) node {-1};
\shade[shading=ball, ball color = Gainsboro] (2.25,3) circle (.25) node {-5};

\begin{scope}[shift={(4,0)}]
\path (0,0) to (0,4);
\path[shift = {(.75,0)}] (0,0) to (0,4);
\path[shift = {(1.5,0)}] (0,0) to (0,4);
\path[shift = {(2.25,0)}] (0,0) to (0,4);
\shade[shading=ball, ball color = Gainsboro] (.75,1) circle (.25) node {9};
\shade[shading=ball, ball color = Gainsboro] (.75,1.5) circle (.25) node{5};
\shade[shading=ball, ball color = Gainsboro] (.75,2) circle (.25) node {1};
\shade[shading=ball, ball color = Gainsboro] (.75,2.5) circle (.25) node {-3};
\shade[shading=ball, ball color = Gainsboro] (.75,3) circle (.25) node {-7};
\fill[fill=white] (1.5,1) circle (.25) node {10};
\fill[fill=white] (1.5,1.5) circle (.25) node{6};
\shade[shading=ball, ball color = Gainsboro] (1.5,2) circle (.25) node {2};
\shade[shading=ball, ball color = Gainsboro] (1.5,2.5) circle (.25) node {-2};
\shade[shading=ball, ball color = Gainsboro] (1.5,3) circle (.25) node {-6};

\fill[fill=white] (2.25,1) circle (.25) node {11};
\fill[fill=white] (2.25,1.5) circle (.25) node {7};
\shade[shading=ball, ball color = Gainsboro] (2.25,2) circle (.25) node {3};
\shade[shading=ball, ball color = Gainsboro] (2.25,2.5) circle (.25) node {-1};
\shade[shading=ball, ball color = Gainsboro] (2.25,3) circle (.25) node {-5};

\fill[fill=white] (0,1) circle (.25) node {8};
\fill[fill=white](0,1.5) circle (.25) node {4};
\fill[fill=white] (0,2) circle (.25) node {0};
 \shade[shading=ball, ball color = Gainsboro] (0,2.5) circle (.25) node {-4};
\shade[shading=ball, ball color = Gainsboro] (0,3) circle (.25) node {-8};
\end{scope}
\begin{scope}[shift={(3,-.5)}]
\node[shift = {(-1,-1)}] {$\lambda =$};
\def\boxpath{-- +(-0.5,0) -- +(-0.5,-0.5) -- +(0,-0.5) -- cycle};
\draw[step =0.5,shift={(0,0)}, thick](0,-0.5) \boxpath node[anchor= south east ]{9};
\draw[step =0.5,shift={(0,0)}, thick](0.5,-0.5) \boxpath node[anchor= south east ]{5};
\draw[step =0.5,shift={(0,0)}, thick](1,-0.5) \boxpath node[anchor= south east ]{3};
\draw[step =0.5,shift={(0,0)}, thick](1.5,-0.5) \boxpath node[anchor= south east ]{2};
\draw[step =0.5,shift={(0,0)}, thick](2,-0.5) \boxpath node[anchor= south east ]{1};
\draw[step =0.5,shift={(0,0)}, thick](0,-1) \boxpath node[anchor= south east ]{5};
\draw[step =0.5,shift={(0,0)}, thick](0.5,-1) \boxpath node[anchor= south east ]{1};
\draw[step =0.5,shift={(0,0)}, thick](0,-1.5) \boxpath node[anchor= south east ]{3};
\draw[step =0.5,shift={(0,0)}, thick](0,-2) \boxpath node[anchor= south east ]{2};
\draw[step =0.5,shift={(0,0)}, thick](0,-2.5) \boxpath node[anchor= south east ]{1};

\end{scope}
\end{tikzpicture}
\caption{\small     Both abaci represent the $4$-core $\lambda$.  }
  \label{fig;abacus}
\end{figure}

\end{example}
Given a flush abacus, that is, the abacus of an $\n$-core $\lambda$, we
can associate to it the vector whose $i^{\mathrm th}$ entry is
the largest level in runner $i-1$ which contains a bead.
Note that the sum of the entries in this vector is the balance number of the
abacus.
When the abacus is balanced,
we will call this vector $\vec n (\lambda)$, in keeping
with the notation of \cite{GKS;cranks}.  We note that $\vec n (\lambda) \in Q$.
\begin{example}
In the example above $\vec n (\lambda) =(2,0,0,-2)$,
and the vector for the unbalanced abacus is $(-1,2,0,0)$.
\end{example}
We recall the following claim, which can be found in \cite{BJV}.
\begin{claim}
The map $\lambda \mapsto \vec n (\lambda) $ is an $\affS$-equivariant bijection
$\{ \n \text{-cores} \} \to Q$.
\end{claim}

We recall here results of Anderson \cite{Anderson}, 
which describe the abacus of an $n$-core
that is also a $t$-core, for $t$ relatively prime to $n$.
When $t=\m\n+1$, this takes a particularly nice form.

\begin{proposition}[Anderson] \label{prop;anderson}
Let $\lambda$ be an $n$-core.
Suppose $t$ is relatively prime to $\n$. Let $M = \n t -\n -t$. 
Consider the grid of points $(x,y)
\in \Z \times \Z$
with $0 \le x \le \n-1$, $0 \le y $
labelled  by $M - x t - y \n $. 
Circle a point in this grid if and only if its label  is
obtained from the first column hooklengths 
of $\lambda$. 
Then $\lambda$ is a $t$-core if and only if 
\begin{enumerate}
\item
All beads in  the abacus of $\lambda$ are at entries $\le M$,
in other words at $(x,y)$ with $0 \le x \le \n-1$, $0 \le y $;
\item
The  circled points in the grid are upwards flush,
in other words if $(x,y)$  is circled, so is $(x, y-1)$;
\item
The  circled points in the grid are flush to the right,
in other words if $(x,y)$  is circled and $x \le \n-2$, so is $(x+1, y)$.
\end{enumerate}
\end{proposition}

Note that the columns of this grid are exactly the runners of $\lambda$'s 
abacus, written out of order,
with each runner shifted up or down  relative to its new left neighbor.
This shifting is performed exactly so labels in the same row are congruent
$\mod t$.  This explains why the circles must be flush to the right
as well as upwards flush.

In the special case $t = \m\n+1$, the columns of the grid are
the runners of $\lambda$'s abacus, written  in reverse order.  
Furthermore,
 each runner has been  shifted $\m$ units down relative to its new left neighbor.
So the condition of being flush to the right on Anderson's grid is given by requiring
on the abacus
that if the largest circled entry on runner $i+1$ is at level $r$
then runner $i$ must have a circled entry at level $r-\m$.
In other words, if $(a_1, \ldots, a_n) = \vec n (\lambda)$, then we require
$a_{i} +m -  a_{i+1} \ge 0$, i.e. 
$\brac{\vec n (\lambda)}{\alpha_i} \ge -\m$ for $0 < i<\n$. 
Recall the $0^{\mathrm th}$ and  $(n-1)^{\mathrm st}$ and  runners
must also have this relationship (adding a constant to all entries in
the abacus cyclically permutes the runners).
This condition becomes $a_{n}+1 +m - a_1 \ge 0$, i.e.
$\brac{\vec n (\lambda)}{\theta} \le \m+1$.

\begin{corollary}
\label{cor;anderson}
Let $\lambda$ be an $\n$-core. 
Then $\lambda$ is an $(\m\n+1)$-core if and only if 
$\brac{\vec n (\lambda)}{\alpha_i} \ge -\m$ for $0 < i<\n$ and
$\brac{\vec n (\lambda)}{\theta} \le \m+1$.
\end{corollary}


\section{The bijection between cores and alcoves}
\label{sec-bij}

In this section, we will describe a bijection between the set of partitions
that are both $\n$-cores and $(\m\n+1)$-cores and the connected components
of the $\m$-Shi hyperplane arrangement complement that lie in the dominant chamber,
or more specifically, the dominant $\m$-minimal alcoves.
Furthermore, this bijection commutes with the action of $\affS$.
(We note the minor technicality that the action on cores is a left action, but
we take the right action on alcoves when discussing the Shi arrangement.)

In particular, this map is just the restriction of the $\affS$-equivariant map
\begin{align*}
\{ \n \text{-cores} \} & \to \{ \text{alcoves in the dominant chamber} \}
\\
 w \emptyset &\mapsto  w^{-1} \Afund.
\end{align*}

\begin{theorem}
\label{thm;main}
The map $\Phi: w  \emptyset  \mapsto w^{-1} \Afund$  for 
$w $  a minimal length left coset representative of $ \affS/\Sn$
induces a bijection between
the set of $\n$-cores that are also $(\m \n +1)$-cores and the set of
$\m$-minimal alcoves  in the dominant chamber.
\end{theorem}

\begin{proof}
Let $\lambda$ be an $\n$-core  and write $\lambda = w \emptyset$ for $w \in \affS$
a minimal length left coset representative for $\affS / \Sn$.
Recall that $\vec n (\lambda) = w  (0, 0, \ldots, 0) \in Q$. 
Note, that since $w$ is a minimal length left coset representative,
$w^{-1}$ is a minimal length right coset representative and in particular,
$w^{-1} \Afund$ is in the dominant chamber.
Recall by Corollary \ref{cor;minimal} that in this case
$w^{-1} \Afund$ is $\m$-minimal
if and only if
$\brac{\beta}{\alpha_i} \ge -\m$ for  $0<i<n$ 
and $\brac{\beta}{\theta} \le m+1$, where
$\beta =  w (0, \ldots , 0) = \vec n (\lambda)$.

In Corollary \ref{cor;anderson} above, 
we have $\lambda$ is an $(\m\n+1)$-core iff 
the conditions above hold for $\beta = \vec n (\lambda)$.
\end{proof}

The bijection is pictured below in Figure \ref{fig;core}.

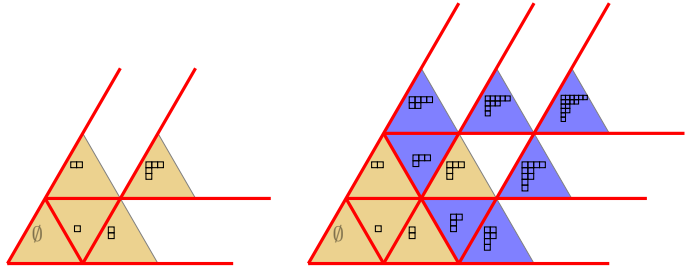
\begin{figure}[ht]

\begin{tikzpicture}[fill opacity=.5]
\tikzstyle{every node}=[font=\footnotesize]
[0,0,0]
\filldraw [draw=black,thick,fill opacity = .5,help lines, fill = Goldenrod,shift={(0.000000,0)}](0,0) -- (60:1cm) -- (0:1cm) --(0,0) node[blue] at (0.4cm,0.2cm) {} node[black] at (0.4cm,0.4cm) {\tiny{$\emptyset$}} node[blue] at (0.4cm,0.6cm) {};
\def\boxpath{-- +(-0.0773503,0) -- +(-0.0773503,-0.0773503) -- +(0,-0.0773503) -- cycle};
[1,0,0]
\filldraw [draw=black, thick,help lines,fill = Goldenrod, shift={(1,0)}](0,0) -- (60:1cm) -- (120:1cm) --(0,0) node[green,fill opacity=1] at (0cm,0.8cm) {} node[black,fill opacity=1, draw opacity = 1] at (0cm,0.6cm) {} node[green,fill opacity=1] at (0cm,0.4cm) {};
\def\boxpath{-- +(-0.0773503,0) -- +(-0.0773503,-0.0773503) -- +(0,-0.0773503) -- cycle};
\draw[step =0.0773503,shift={(0.967949,0.579274)}](0,-0.0773503) \boxpath;
[1,0,1]
\filldraw [draw=black,thick,fill opacity = .5,help lines, fill = Goldenrod,shift={(1.000000,0)}](0,0) -- (60:1cm) -- (0:1cm) --(0,0) node[blue] at (0.4cm,0.2cm) {} node[black] at (0.4cm,0.4cm) {} node[blue] at (0.4cm,0.6cm) {};
\def\boxpath{-- +(-0.0773503,0) -- +(-0.0773503,-0.0773503) -- +(0,-0.0773503) -- cycle};
\draw[step =0.0773503,shift={(1.41795,0.564102)}](0,-0.0773503) \boxpath;
\draw[step =0.0773503,shift={(1.41795,0.564102)}](0,-0.154701) \boxpath;
[1,1,0]
\filldraw [draw=black,thick,fill opacity = .5,help lines, fill = Goldenrod,shift={(0.500000,0.866025)}](0,0) -- (60:1cm) -- (0:1cm) --(0,0) node[blue] at (0.4cm,0.2cm) {} node[black] at (0.4cm,0.4cm) {} node[blue] at (0.4cm,0.6cm) {};
\def\boxpath{-- +(-0.0773503,0) -- +(-0.0773503,-0.0773503) -- +(0,-0.0773503) -- cycle};
\draw[step =0.0773503,shift={(0.917949,1.43013)}](0,-0.0773503) \boxpath;
\draw[step =0.0773503,shift={(0.917949,1.43013)}](0.0773503,-0.0773503) \boxpath;
[2,1,1]
\filldraw [draw=black,thick,fill opacity = .5,help lines, fill = Goldenrod,shift={(1.500000,0.866025)}](0,0) -- (60:1cm) -- (0:1cm) --(0,0) node[blue] at (0.4cm,0.2cm) {} node[black] at (0.4cm,0.4cm) {} node[blue] at (0.4cm,0.6cm) {};
\def\boxpath{-- +(-0.0773503,0) -- +(-0.0773503,-0.0773503) -- +(0,-0.0773503) -- cycle};
\draw[step =0.0773503,shift={(1.91795,1.43013)}](0,-0.0773503) \boxpath;
\draw[step =0.0773503,shift={(1.91795,1.43013)}](0.0773503,-0.0773503) \boxpath;
\draw[step =0.0773503,shift={(1.91795,1.43013)}](0.154701,-0.0773503) \boxpath;
\draw[step =0.0773503,shift={(1.91795,1.43013)}](0,-0.154701) \boxpath;
\draw[step =0.0773503,shift={(1.91795,1.43013)}](0,-0.232051) \boxpath;
\path [draw = red, very thick, draw opacity = 1, shift = {(0,0)}] (0,0) -- +(60:3);
\path [draw = red,very thick, draw opacity = 1] (0,0) -- +(3,0);
\path [draw = red,very thick,draw opacity = 1] (0,0) -- (0,0);
\path [draw = red, very thick, draw opacity = 1, shift = {(1,0)}] (0,0) -- +(60:3);
\path [draw = red,very thick, draw opacity = 1] (0.5,0.866025) -- +(3,0);
\path [draw = red,very thick,draw opacity = 1] (0.5,0.866025) -- (1,0);
[0,0,0]
\filldraw [draw=black,thick,fill opacity = .5,help lines, fill = Goldenrod,shift={(4.000000,0)}](0,0) -- (60:1cm) -- (0:1cm) --(0,0) node[blue] at (0.4cm,0.2cm) {} node[black] at (0.4cm,0.4cm) {\tiny{$\emptyset$}} node[blue] at (0.4cm,0.6cm) {};
\def\boxpath{-- +(-0.0773503,0) -- +(-0.0773503,-0.0773503) -- +(0,-0.0773503) -- cycle};
[1,0,0]
\filldraw [draw=black, thick,help lines,fill = Goldenrod, shift={(5,0)}](0,0) -- (60:1cm) -- (120:1cm) --(0,0) node[green,fill opacity=1] at (0cm,0.8cm) {} node[black,fill opacity=1, draw opacity = 1] at (0cm,0.6cm) {} node[green,fill opacity=1] at (0cm,0.4cm) {};
\def\boxpath{-- +(-0.0773503,0) -- +(-0.0773503,-0.0773503) -- +(0,-0.0773503) -- cycle};
\draw[step =0.0773503,shift={(4.96795,0.579274)}](0,-0.0773503) \boxpath;
[1,0,1]
\filldraw [draw=black,thick,fill opacity = .5,help lines, fill = Goldenrod,shift={(5.000000,0)}](0,0) -- (60:1cm) -- (0:1cm) --(0,0) node[blue] at (0.4cm,0.2cm) {} node[black] at (0.4cm,0.4cm) {} node[blue] at (0.4cm,0.6cm) {};
\def\boxpath{-- +(-0.0773503,0) -- +(-0.0773503,-0.0773503) -- +(0,-0.0773503) -- cycle};
\draw[step =0.0773503,shift={(5.41795,0.564102)}](0,-0.0773503) \boxpath;
\draw[step =0.0773503,shift={(5.41795,0.564102)}](0,-0.154701) \boxpath;
[2,0,1]
\filldraw [draw=black, thick,help lines,fill = blue, shift={(6,0)}](0,0) -- (60:1cm) -- (120:1cm) --(0,0) node[green,fill opacity=1] at (0cm,0.8cm) {} node[black,fill opacity=1, draw opacity = 1] at (0cm,0.6cm) {} node[green,fill opacity=1] at (0cm,0.4cm) {};
\def\boxpath{-- +(-0.0773503,0) -- +(-0.0773503,-0.0773503) -- +(0,-0.0773503) -- cycle};
\draw[step =0.0773503,shift={(5.96795,0.733975)}](0,-0.0773503) \boxpath;
\draw[step =0.0773503,shift={(5.96795,0.733975)}](0.0773503,-0.0773503) \boxpath;
\draw[step =0.0773503,shift={(5.96795,0.733975)}](0,-0.154701) \boxpath;
\draw[step =0.0773503,shift={(5.96795,0.733975)}](0,-0.232051) \boxpath;
[2,0,2]
\filldraw [draw=black,thick,fill opacity = .5,help lines, fill = blue,shift={(6.000000,0)}](0,0) -- (60:1cm) -- (0:1cm) --(0,0) node[blue] at (0.4cm,0.2cm) {} node[black] at (0.4cm,0.4cm) {} node[blue] at (0.4cm,0.6cm) {};
\def\boxpath{-- +(-0.0773503,0) -- +(-0.0773503,-0.0773503) -- +(0,-0.0773503) -- cycle};
\draw[step =0.0773503,shift={(6.41795,0.564102)}](0,-0.0773503) \boxpath;
\draw[step =0.0773503,shift={(6.41795,0.564102)}](0.0773503,-0.0773503) \boxpath;
\draw[step =0.0773503,shift={(6.41795,0.564102)}](0,-0.154701) \boxpath;
\draw[step =0.0773503,shift={(6.41795,0.564102)}](0.0773503,-0.154701) \boxpath;
\draw[step =0.0773503,shift={(6.41795,0.564102)}](0,-0.232051) \boxpath;
\draw[step =0.0773503,shift={(6.41795,0.564102)}](0,-0.309401) \boxpath;
[1,1,0]
\filldraw [draw=black,thick,fill opacity = .5,help lines, fill = Goldenrod,shift={(4.500000,0.866025)}](0,0) -- (60:1cm) -- (0:1cm) --(0,0) node[blue] at (0.4cm,0.2cm) {} node[black] at (0.4cm,0.4cm) {} node[blue] at (0.4cm,0.6cm) {};
\def\boxpath{-- +(-0.0773503,0) -- +(-0.0773503,-0.0773503) -- +(0,-0.0773503) -- cycle};
\draw[step =0.0773503,shift={(4.91795,1.43013)}](0,-0.0773503) \boxpath;
\draw[step =0.0773503,shift={(4.91795,1.43013)}](0.0773503,-0.0773503) \boxpath;
[2,1,0]
\filldraw [draw=black, thick,help lines,fill = blue, shift={(5.5,0.866025)}](0,0) -- (60:1cm) -- (120:1cm) --(0,0) node[green,fill opacity=1] at (0cm,0.8cm) {} node[black,fill opacity=1, draw opacity = 1] at (0cm,0.6cm) {} node[green,fill opacity=1] at (0cm,0.4cm) {};
\def\boxpath{-- +(-0.0773503,0) -- +(-0.0773503,-0.0773503) -- +(0,-0.0773503) -- cycle};
\draw[step =0.0773503,shift={(5.46795,1.52265)}](0,-0.0773503) \boxpath;
\draw[step =0.0773503,shift={(5.46795,1.52265)}](0.0773503,-0.0773503) \boxpath;
\draw[step =0.0773503,shift={(5.46795,1.52265)}](0.154701,-0.0773503) \boxpath;
\draw[step =0.0773503,shift={(5.46795,1.52265)}](0,-0.154701) \boxpath;
[2,1,1]
\filldraw [draw=black,thick,fill opacity = .5,help lines, fill = Goldenrod,shift={(5.500000,0.866025)}](0,0) -- (60:1cm) -- (0:1cm) --(0,0) node[blue] at (0.4cm,0.2cm) {} node[black] at (0.4cm,0.4cm) {} node[blue] at (0.4cm,0.6cm) {};
\def\boxpath{-- +(-0.0773503,0) -- +(-0.0773503,-0.0773503) -- +(0,-0.0773503) -- cycle};
\draw[step =0.0773503,shift={(5.91795,1.43013)}](0,-0.0773503) \boxpath;
\draw[step =0.0773503,shift={(5.91795,1.43013)}](0.0773503,-0.0773503) \boxpath;
\draw[step =0.0773503,shift={(5.91795,1.43013)}](0.154701,-0.0773503) \boxpath;
\draw[step =0.0773503,shift={(5.91795,1.43013)}](0,-0.154701) \boxpath;
\draw[step =0.0773503,shift={(5.91795,1.43013)}](0,-0.232051) \boxpath;
[3,1,2]
\filldraw [draw=black,thick,fill opacity = .5,help lines, fill = blue,shift={(6.500000,0.866025)}](0,0) -- (60:1cm) -- (0:1cm) --(0,0) node[blue] at (0.4cm,0.2cm) {} node[black] at (0.4cm,0.4cm) {} node[blue] at (0.4cm,0.6cm) {};
\def\boxpath{-- +(-0.0773503,0) -- +(-0.0773503,-0.0773503) -- +(0,-0.0773503) -- cycle};
\draw[step =0.0773503,shift={(6.91795,1.43013)}](0,-0.0773503) \boxpath;
\draw[step =0.0773503,shift={(6.91795,1.43013)}](0.0773503,-0.0773503) \boxpath;
\draw[step =0.0773503,shift={(6.91795,1.43013)}](0.154701,-0.0773503) \boxpath;
\draw[step =0.0773503,shift={(6.91795,1.43013)}](0.232051,-0.0773503) \boxpath;
\draw[step =0.0773503,shift={(6.91795,1.43013)}](0,-0.154701) \boxpath;
\draw[step =0.0773503,shift={(6.91795,1.43013)}](0.0773503,-0.154701) \boxpath;
\draw[step =0.0773503,shift={(6.91795,1.43013)}](0,-0.232051) \boxpath;
\draw[step =0.0773503,shift={(6.91795,1.43013)}](0.0773503,-0.232051) \boxpath;
\draw[step =0.0773503,shift={(6.91795,1.43013)}](0,-0.309401) \boxpath;
\draw[step =0.0773503,shift={(6.91795,1.43013)}](0,-0.386751) \boxpath;
[2,2,0]
\filldraw [draw=black,thick,fill opacity = .5,help lines, fill = blue,shift={(5.000000,1.73205)}](0,0) -- (60:1cm) -- (0:1cm) --(0,0) node[blue] at (0.4cm,0.2cm) {} node[black] at (0.4cm,0.4cm) {} node[blue] at (0.4cm,0.6cm) {};
\def\boxpath{-- +(-0.0773503,0) -- +(-0.0773503,-0.0773503) -- +(0,-0.0773503) -- cycle};
\draw[step =0.0773503,shift={(5.41795,2.29615)}](0,-0.0773503) \boxpath;
\draw[step =0.0773503,shift={(5.41795,2.29615)}](0.0773503,-0.0773503) \boxpath;
\draw[step =0.0773503,shift={(5.41795,2.29615)}](0.154701,-0.0773503) \boxpath;
\draw[step =0.0773503,shift={(5.41795,2.29615)}](0.232051,-0.0773503) \boxpath;
\draw[step =0.0773503,shift={(5.41795,2.29615)}](0,-0.154701) \boxpath;
\draw[step =0.0773503,shift={(5.41795,2.29615)}](0.0773503,-0.154701) \boxpath;
[3,2,1]
\filldraw [draw=black,thick,fill opacity = .5,help lines, fill = blue,shift={(6.000000,1.73205)}](0,0) -- (60:1cm) -- (0:1cm) --(0,0) node[blue] at (0.4cm,0.2cm) {} node[black] at (0.4cm,0.4cm) {} node[blue] at (0.4cm,0.6cm) {};
\def\boxpath{-- +(-0.0663002,0) -- +(-0.0663002,-0.0663002) -- +(0,-0.0663002) -- cycle};
\draw[step =0.0663002,shift={(6.41795,2.29615)}](0,-0.0663002) \boxpath;
\draw[step =0.0663002,shift={(6.41795,2.29615)}](0.0663002,-0.0663002) \boxpath;
\draw[step =0.0663002,shift={(6.41795,2.29615)}](0.1326,-0.0663002) \boxpath;
\draw[step =0.0663002,shift={(6.41795,2.29615)}](0.198901,-0.0663002) \boxpath;
\draw[step =0.0663002,shift={(6.41795,2.29615)}](0.265201,-0.0663002) \boxpath;
\draw[step =0.0663002,shift={(6.41795,2.29615)}](0,-0.1326) \boxpath;
\draw[step =0.0663002,shift={(6.41795,2.29615)}](0.0663002,-0.1326) \boxpath;
\draw[step =0.0663002,shift={(6.41795,2.29615)}](0.1326,-0.1326) \boxpath;
\draw[step =0.0663002,shift={(6.41795,2.29615)}](0,-0.198901) \boxpath;
\draw[step =0.0663002,shift={(6.41795,2.29615)}](0,-0.265201) \boxpath;
[4,2,2]
\filldraw [draw=black,thick,fill opacity = .5,help lines, fill = blue,shift={(7.000000,1.73205)}](0,0) -- (60:1cm) -- (0:1cm) --(0,0) node[blue] at (0.4cm,0.2cm) {} node[black] at (0.4cm,0.4cm) {} node[blue] at (0.4cm,0.6cm) {};
\def\boxpath{-- +(-0.0580127,0) -- +(-0.0580127,-0.0580127) -- +(0,-0.0580127) -- cycle};
\draw[step =0.0580127,shift={(7.41795,2.29615)}](0,-0.0580127) \boxpath;
\draw[step =0.0580127,shift={(7.41795,2.29615)}](0.0580127,-0.0580127) \boxpath;
\draw[step =0.0580127,shift={(7.41795,2.29615)}](0.116025,-0.0580127) \boxpath;
\draw[step =0.0580127,shift={(7.41795,2.29615)}](0.174038,-0.0580127) \boxpath;
\draw[step =0.0580127,shift={(7.41795,2.29615)}](0.232051,-0.0580127) \boxpath;
\draw[step =0.0580127,shift={(7.41795,2.29615)}](0.290064,-0.0580127) \boxpath;
\draw[step =0.0580127,shift={(7.41795,2.29615)}](0,-0.116025) \boxpath;
\draw[step =0.0580127,shift={(7.41795,2.29615)}](0.0580127,-0.116025) \boxpath;
\draw[step =0.0580127,shift={(7.41795,2.29615)}](0.116025,-0.116025) \boxpath;
\draw[step =0.0580127,shift={(7.41795,2.29615)}](0.174038,-0.116025) \boxpath;
\draw[step =0.0580127,shift={(7.41795,2.29615)}](0,-0.174038) \boxpath;
\draw[step =0.0580127,shift={(7.41795,2.29615)}](0.0580127,-0.174038) \boxpath;
\draw[step =0.0580127,shift={(7.41795,2.29615)}](0,-0.232051) \boxpath;
\draw[step =0.0580127,shift={(7.41795,2.29615)}](0.0580127,-0.232051) \boxpath;
\draw[step =0.0580127,shift={(7.41795,2.29615)}](0,-0.290064) \boxpath;
\draw[step =0.0580127,shift={(7.41795,2.29615)}](0,-0.348076) \boxpath;
\path [draw = red, very thick, draw opacity = 1, shift = {(4,0)}] (0,0) -- +(60:4);
\path [draw = red,very thick, draw opacity = 1] (4,0) -- +(4,0);
\path [draw = red,very thick,draw opacity = 1] (4,0) -- (4,0);
\path [draw = red, very thick, draw opacity = 1, shift = {(5,0)}] (0,0) -- +(60:4);
\path [draw = red,very thick, draw opacity = 1] (4.5,0.866025) -- +(4,0);
\path [draw = red,very thick,draw opacity = 1] (4.5,0.866025) -- (5,0);
\path [draw = red, very thick, draw opacity = 1, shift = {(6,0)}] (0,0) -- +(60:4);
\path [draw = red,very thick, draw opacity = 1] (5,1.73205) -- +(4,0);
\path [draw = red,very thick,draw opacity = 1] (5,1.73205) -- (6,0);
\end{tikzpicture}
  \caption{\small 
$\m$-minimal alcoves $w^{-1} \Afund$ in the dominant chamber
of the  $1$-Shi and $2$-Shi  arrangements
of type $A_2$
filled with the $3$-core partition $w \emptyset$.
In the first, they are also $4$-cores, and the second are also $7$-cores.
}
  \label{fig;core}
\end{figure}
%

\section{A bijection on alcoves}
\label{sec-inverse}

Although it is not an ingredient in the main theorem of this paper,
the following theorem builds on the work of Sections \ref{sec-minimal} and
 \ref{sec-inversion}.
We thank Mark Haiman for pointing it out to us.

Note that
the set of alcoves in $\Region$ is in bijection with $Q / (\m\n  +1)Q$.
Furthermore, it is easy to see by translating
(by $m \rho = \frac {m}{2} \sum_{\alpha \in \pos} \alpha$)
that $Q \cap \Region$ is in bijection with $Q \cap \overline{(\m\n + 1)\Afund}$.
It is the latter that is discussed in Lemma 7.4.1 of \cite{Haiman}
and  studied in \cite{A2005a} (technically for the co-root
lattice $Q^\vee$).
Taking the latter bijection into account, 
 the second statement of Theorem \ref{thm;haiman} below appears in Theorem 4.2 of
\cite{A2005a}.

\begin{theorem}
\label{thm;haiman}
\begin{enumerate}
\item
The map $w \Afund \mapsto w^{-1} \Afund$ restricts to a bijection between
alcoves in the region
$\Region=\{ v \in V \mid \brac{v}{\alpha_i} \ge -m, \brac{v}{\theta} \le m+1 \}$
and $\m$-minimal alcoves.
\item
The map $w  (0, \ldots, 0)  \mapsto w^{-1} \Afund$ restricts to a bijection between
$Q \cap \Region$ and
$\m$-minimal alcoves  in the dominant chamber.
\end{enumerate}
\end{theorem}

\begin{proof}
Observe $\Region
= \Hn{\theta}{m+1}  \cap \bigcap_{i=1}^{n-1} \Hp{\alpha_i}{-m}$ 
 can be viewed as an $\m$-dilation of (the closure of) $\Afund \subseteq
 \Hn{\theta}{1}  \cap \bigcap_{i=1}^{n-1} \Hp{\alpha_i}{0}$.

 The second statement follows directly from  Corollary \ref{cor;minimal}.

A proof of the first  statement can be given that 
is very similar to that of Propositions
\ref{prop;hyper} and \ref{prop;inv}.
Instead we shall use those propositions to prove it.

Write 
$w^{-1}(\alpha_i) = \alpha -k \delta$ with $\alpha \in\roots$.

Suppose  $k < 0$.  Then $\alpha_i \not\in \Inv(w^{-1})$ so by Proposition \ref{prop;inv}
$w \Afund \subseteq \Hn{-\alpha_i}{0}  = \Hp{\alpha_i}{0} \subseteq \Hp{\alpha_i}{-m}$.

Suppose
$k \ge 0$. Then the $\m$-minimality of $w^{-1} \Afund$
implies $k \le m$ when $\alpha > 0$ and $k \le m-1$ in the case $\alpha < 0$ by
Propositions
 \ref{prop;inv} and  \ref{prop;hyper} and Remark \ref{remark;min}.
When $\alpha > 0$, $w^{-1}(\alpha_i + k \delta) = \alpha$ so $\alpha_i + k \delta
\not\in \Inv(w^{-1})$. By Proposition \ref{prop;inv},
$w \Afund \subseteq \Hn{-\alpha_i}{k}  = \Hp{\alpha_i}{-k} \subseteq \Hp{\alpha_i}{-m}$
since $-k \ge -m$.
When $\alpha < 0$, $w^{-1}(\alpha_i + (k+1) \delta) = \alpha + \delta > 0$
so
$w \Afund \subseteq \Hn{-\alpha_i}{k+1}  = \Hp{\alpha_i}{-k-1} \subseteq \Hp{\alpha_i}{-m}$
since $-k -1\ge -m$.
Hence $w \Afund \subseteq \bigcap_{i=1}^{n-1} \Hp{\alpha_i}{-m}$.

In the case $i=0$ since $\alpha_i = \delta - \theta$, a similar argument gives
$w \Afund \subseteq \Hn{\theta}{m+1}$.  Hence we get 
$w \Afund \subseteq \Region$.

For the converse,
suppose $w \Afund \subseteq \Region$.
Letting $\beta = w(0,0,\ldots, 0) \in \Region$ we get 
$\brac{\beta}{\theta} \le m+1, \brac{\beta}{\alpha_i} \ge -m$, $1 \le i < n$
because $\Region$ is closed. 
Assume $w^{-1} \Afund$ is not $\m$-minimal.
Then $\exists i, \alpha, k$ with $ 0 \le i < n, \alpha \in \roots, k \in\Z$ such that
$w^{-1} \Afund \subseteq \Hp{\alpha}{k}$ but 
$w^{-1} s_i \Afund \subseteq \Hn{\alpha}{k}$
with $k > m$ if $\alpha > 0$, but $k > m-1$ if $\alpha < 0$.
Note by Proposition \ref{prop;hyper} $w^{-1}(\alpha_i) = \alpha - k\delta$
where $-k=\brac{\beta}{\alpha_i}$.
First consider $\alpha >0$.  Then $\brac{\beta}{\alpha_i}  = -k < -m$ which
is a contradiction.
Next consider $\alpha < 0$.  Then $\alpha_i + k\delta \in \Inv(w^{-1})$
so by Proposition \ref{prop;inv}
$w \Afund \subseteq \Hp{-\alpha_i}{k} = \Hn{\alpha_i}{-k} \subseteq \Hn{\alpha_i}{-m} $
since $k \ge m > m-1$, contradicting $w \Afund \subseteq \Region \subseteq 
\Hp{\alpha_i}{-m} .$
\end{proof}

The bijection is illustrated below, the first part comparing
 Figure \ref{fig;onetwoshi} to Figure \ref{fig;dilate}, and the second
part in Figure \ref{fig;rightleft}.

\begin{figure}[ht]

\begin{tikzpicture}[fill opacity=.5]
\tikzstyle{every node}=[font=\footnotesize]
[-2,-1,-1]
\filldraw [draw=black,thick,fill opacity = .5, fill = Goldenrod,shift={(-1.125000,-0.649519)}](0,0) -- (60:0.75cm) -- (0:0.75cm) --(0,0) node[blue] at (0.3cm,0.15cm) {} node[black] at (0.3cm,0.3cm) {} node[blue] at (0.3cm,0.45cm) {};
\begin{pgfonlayer}{foreground layer}
\fill[fill=Goldenrod, fill opacity=1.](-1.125,-0.649519) circle (0.0675cm);
\draw[black, fill opacity=1.](-1.125,-0.649519) circle (0.0825cm);
\end{pgfonlayer}{foreground layer}
[-1,-1,-1]
\filldraw [draw=black, thick,fill = Gainsboro, shift={(-0.375,-0.649519)}](0,0) -- (60:0.75cm) -- (120:0.75cm) --(0,0) node[green,fill opacity=1] at (0cm,0.6cm) {} node[black,fill opacity=1, draw opacity = 1] at (0cm,0.45cm) {} node[green,fill opacity=1] at (0cm,0.3cm) {};
[-1,-1,0]
\filldraw [draw=black,thick,fill opacity = .5, fill = Gainsboro,shift={(-0.375000,-0.649519)}](0,0) -- (60:0.75cm) -- (0:0.75cm) --(0,0) node[blue] at (0.3cm,0.15cm) {} node[black] at (0.3cm,0.3cm) {} node[blue] at (0.3cm,0.45cm) {};
[0,-1,0]
\filldraw [draw=black, thick,fill = Gainsboro, shift={(0.375,-0.649519)}](0,0) -- (60:0.75cm) -- (120:0.75cm) --(0,0) node[green,fill opacity=1] at (0cm,0.6cm) {} node[black,fill opacity=1, draw opacity = 1] at (0cm,0.45cm) {} node[green,fill opacity=1] at (0cm,0.3cm) {};
[0,-1,1]
\filldraw [draw=black,thick,fill opacity = .5, fill = Goldenrod,shift={(0.375000,-0.649519)}](0,0) -- (60:0.75cm) -- (0:0.75cm) --(0,0) node[blue] at (0.3cm,0.15cm) {} node[black] at (0.3cm,0.3cm) {} node[blue] at (0.3cm,0.45cm) {};
\begin{pgfonlayer}{foreground layer}
\fill[fill=Goldenrod, fill opacity=1.](1.125,-0.649519) circle (0.0675cm);
\draw[black, fill opacity=1.](1.125,-0.649519) circle (0.0825cm);
\end{pgfonlayer}{foreground layer}
[1,-1,1]
\filldraw [draw=black, thick,fill = Gainsboro, shift={(1.125,-0.649519)}](0,0) -- (60:0.75cm) -- (120:0.75cm) --(0,0) node[green,fill opacity=1] at (0cm,0.6cm) {} node[black,fill opacity=1, draw opacity = 1] at (0cm,0.45cm) {} node[green,fill opacity=1] at (0cm,0.3cm) {};
[1,-1,2]
\filldraw [draw=black,thick,fill opacity = .5, fill = Gainsboro,shift={(1.125000,-0.649519)}](0,0) -- (60:0.75cm) -- (0:0.75cm) --(0,0) node[blue] at (0.3cm,0.15cm) {} node[black] at (0.3cm,0.3cm) {} node[blue] at (0.3cm,0.45cm) {};
[-1,0,-1]
\filldraw [draw=black,thick,fill opacity = .5, fill = Gainsboro,shift={(-0.750000,0)}](0,0) -- (60:0.75cm) -- (0:0.75cm) --(0,0) node[blue] at (0.3cm,0.15cm) {} node[black] at (0.3cm,0.3cm) {} node[blue] at (0.3cm,0.45cm) {};
[0,0,-1]
\filldraw [draw=black, thick,fill = Gainsboro, shift={(0,0)}](0,0) -- (60:0.75cm) -- (120:0.75cm) --(0,0) node[green,fill opacity=1] at (0cm,0.6cm) {} node[black,fill opacity=1, draw opacity = 1] at (0cm,0.45cm) {} node[green,fill opacity=1] at (0cm,0.3cm) {};
[0,0,0]
\filldraw [draw=black,thick,fill opacity = .5, fill = Goldenrod,shift={(0.000000,0)}](0,0) -- (60:0.75cm) -- (0:0.75cm) --(0,0) node[blue] at (0.3cm,0.15cm) {} node[black] at (0.3cm,0.3cm) {} node[blue] at (0.3cm,0.45cm) {};
\begin{pgfonlayer}{foreground layer}
\fill[fill=Goldenrod, fill opacity=1.](0,0) circle (0.0675cm);
\draw[black, fill opacity=1.](0,0) circle (0.0825cm);
\end{pgfonlayer}{foreground layer}
[1,0,0]
\filldraw [draw=black, thick,fill = Goldenrod, shift={(0.75,0)}](0,0) -- (60:0.75cm) -- (120:0.75cm) --(0,0) node[green,fill opacity=1] at (0cm,0.6cm) {} node[black,fill opacity=1, draw opacity = 1] at (0cm,0.45cm) {} node[green,fill opacity=1] at (0cm,0.3cm) {};
\begin{pgfonlayer}{foreground layer}
\fill[fill=Goldenrod, fill opacity=1.](1.125,0.649519) circle (0.0675cm);
\draw[black, fill opacity=1.](1.125,0.649519) circle (0.0825cm);
\end{pgfonlayer}{foreground layer}
[1,0,1]
\filldraw [draw=black,thick,fill opacity = .5, fill = Gainsboro,shift={(0.750000,0)}](0,0) -- (60:0.75cm) -- (0:0.75cm) --(0,0) node[blue] at (0.3cm,0.15cm) {} node[black] at (0.3cm,0.3cm) {} node[blue] at (0.3cm,0.45cm) {};
[0,1,-1]
\filldraw [draw=black,thick,fill opacity = .5, fill = Goldenrod,shift={(-0.375000,0.649519)}](0,0) -- (60:0.75cm) -- (0:0.75cm) --(0,0) node[blue] at (0.3cm,0.15cm) {} node[black] at (0.3cm,0.3cm) {} node[blue] at (0.3cm,0.45cm) {};
\begin{pgfonlayer}{foreground layer}
\fill[fill=Goldenrod, fill opacity=1.](0,1.29904) circle (0.0675cm);
\draw[black, fill opacity=1.](0,1.29904) circle (0.0825cm);
\end{pgfonlayer}{foreground layer}
[1,1,-1]
\filldraw [draw=black, thick,fill = Gainsboro, shift={(0.375,0.649519)}](0,0) -- (60:0.75cm) -- (120:0.75cm) --(0,0) node[green,fill opacity=1] at (0cm,0.6cm) {} node[black,fill opacity=1, draw opacity = 1] at (0cm,0.45cm) {} node[green,fill opacity=1] at (0cm,0.3cm) {};
[1,1,0]
\filldraw [draw=black,thick,fill opacity = .5, fill = Gainsboro,shift={(0.375000,0.649519)}](0,0) -- (60:0.75cm) -- (0:0.75cm) --(0,0) node[blue] at (0.3cm,0.15cm) {} node[black] at (0.3cm,0.3cm) {} node[blue] at (0.3cm,0.45cm) {};
[1,2,-1]
\filldraw [draw=black,thick,fill opacity = .5, fill = Gainsboro,shift={(0.000000,1.29904)}](0,0) -- (60:0.75cm) -- (0:0.75cm) --(0,0) node[blue] at (0.3cm,0.15cm) {} node[black] at (0.3cm,0.3cm) {} node[blue] at (0.3cm,0.45cm) {};
\path [draw = red, very thick,  shift = {(0,0)}](-1.125,-1.94856) -- (1.125,1.94856);
\path [draw = red,very thick, shift = {(0,0)}](-2.25,0) -- (2.25,0);
\path [draw = red, very thick,  shift = {(0,0)}](-1.125,1.94856) -- (1.125,-1.94856);
\path [draw = red, very thick,  shift = {(0.75,0)}](-1.125,-1.94856) -- (1.125,1.94856);
\path [draw = red,very thick, shift = {(0,0.649519)}](-2.25,0) -- (2.25,0);
\path [draw = red, very thick,  shift = {(0.75,0)}](-1.125,1.94856) -- (1.125,-1.94856);
\fill[fill=black, fill opacity=1.](0.,0.) circle (0.0675cm);[-4,-2,-2]
\filldraw [draw=black,thick,fill opacity = .5, fill = blue,shift={(3.750000,-1.29904)}](0,0) -- (60:0.75cm) -- (0:0.75cm) --(0,0) node[blue] at (0.3cm,0.15cm) {} node[black] at (0.3cm,0.3cm) {} node[blue] at (0.3cm,0.45cm) {};
\begin{pgfonlayer}{foreground layer}
\fill[fill=blue, fill opacity=1.](3.75,-1.29904) circle (0.0675cm);
\draw[black, fill opacity=1.](3.75,-1.29904) circle (0.0825cm);
\end{pgfonlayer}{foreground layer}
[-3,-2,-2]
\filldraw [draw=black, thick,fill = Gainsboro, shift={(4.5,-1.29904)}](0,0) -- (60:0.75cm) -- (120:0.75cm) --(0,0) node[green,fill opacity=1] at (0cm,0.6cm) {} node[black,fill opacity=1, draw opacity = 1] at (0cm,0.45cm) {} node[green,fill opacity=1] at (0cm,0.3cm) {};
[-3,-2,-1]
\filldraw [draw=black,thick,fill opacity = .5, fill = Gainsboro,shift={(4.500000,-1.29904)}](0,0) -- (60:0.75cm) -- (0:0.75cm) --(0,0) node[blue] at (0.3cm,0.15cm) {} node[black] at (0.3cm,0.3cm) {} node[blue] at (0.3cm,0.45cm) {};
[-2,-2,-1]
\filldraw [draw=black, thick,fill = Gainsboro, shift={(5.25,-1.29904)}](0,0) -- (60:0.75cm) -- (120:0.75cm) --(0,0) node[green,fill opacity=1] at (0cm,0.6cm) {} node[black,fill opacity=1, draw opacity = 1] at (0cm,0.45cm) {} node[green,fill opacity=1] at (0cm,0.3cm) {};
[-2,-2,0]
\filldraw [draw=black,thick,fill opacity = .5, fill = Gainsboro,shift={(5.250000,-1.29904)}](0,0) -- (60:0.75cm) -- (0:0.75cm) --(0,0) node[blue] at (0.3cm,0.15cm) {} node[black] at (0.3cm,0.3cm) {} node[blue] at (0.3cm,0.45cm) {};
[-1,-2,0]
\filldraw [draw=black, thick,fill = blue, shift={(6,-1.29904)}](0,0) -- (60:0.75cm) -- (120:0.75cm) --(0,0) node[green,fill opacity=1] at (0cm,0.6cm) {} node[black,fill opacity=1, draw opacity = 1] at (0cm,0.45cm) {} node[green,fill opacity=1] at (0cm,0.3cm) {};
\begin{pgfonlayer}{foreground layer}
\fill[fill=blue, fill opacity=1.](6,-1.29904) circle (0.0675cm);
\draw[black, fill opacity=1.](6,-1.29904) circle (0.0825cm);
\end{pgfonlayer}{foreground layer}
[-1,-2,1]
\filldraw [draw=black,thick,fill opacity = .5, fill = Gainsboro,shift={(6.000000,-1.29904)}](0,0) -- (60:0.75cm) -- (0:0.75cm) --(0,0) node[blue] at (0.3cm,0.15cm) {} node[black] at (0.3cm,0.3cm) {} node[blue] at (0.3cm,0.45cm) {};
[0,-2,1]
\filldraw [draw=black, thick,fill = Gainsboro, shift={(6.75,-1.29904)}](0,0) -- (60:0.75cm) -- (120:0.75cm) --(0,0) node[green,fill opacity=1] at (0cm,0.6cm) {} node[black,fill opacity=1, draw opacity = 1] at (0cm,0.45cm) {} node[green,fill opacity=1] at (0cm,0.3cm) {};
[0,-2,2]
\filldraw [draw=black,thick,fill opacity = .5, fill = Gainsboro,shift={(6.750000,-1.29904)}](0,0) -- (60:0.75cm) -- (0:0.75cm) --(0,0) node[blue] at (0.3cm,0.15cm) {} node[black] at (0.3cm,0.3cm) {} node[blue] at (0.3cm,0.45cm) {};
[1,-2,2]
\filldraw [draw=black, thick,fill = Gainsboro, shift={(7.5,-1.29904)}](0,0) -- (60:0.75cm) -- (120:0.75cm) --(0,0) node[green,fill opacity=1] at (0cm,0.6cm) {} node[black,fill opacity=1, draw opacity = 1] at (0cm,0.45cm) {} node[green,fill opacity=1] at (0cm,0.3cm) {};
[1,-2,3]
\filldraw [draw=black,thick,fill opacity = .5, fill = blue,shift={(7.500000,-1.29904)}](0,0) -- (60:0.75cm) -- (0:0.75cm) --(0,0) node[blue] at (0.3cm,0.15cm) {} node[black] at (0.3cm,0.3cm) {} node[blue] at (0.3cm,0.45cm) {};
\begin{pgfonlayer}{foreground layer}
\fill[fill=blue, fill opacity=1.](8.25,-1.29904) circle (0.0675cm);
\draw[black, fill opacity=1.](8.25,-1.29904) circle (0.0825cm);
\end{pgfonlayer}{foreground layer}
[2,-2,3]
\filldraw [draw=black, thick,fill = Gainsboro, shift={(8.25,-1.29904)}](0,0) -- (60:0.75cm) -- (120:0.75cm) --(0,0) node[green,fill opacity=1] at (0cm,0.6cm) {} node[black,fill opacity=1, draw opacity = 1] at (0cm,0.45cm) {} node[green,fill opacity=1] at (0cm,0.3cm) {};
[2,-2,4]
\filldraw [draw=black,thick,fill opacity = .5, fill = Gainsboro,shift={(8.250000,-1.29904)}](0,0) -- (60:0.75cm) -- (0:0.75cm) --(0,0) node[blue] at (0.3cm,0.15cm) {} node[black] at (0.3cm,0.3cm) {} node[blue] at (0.3cm,0.45cm) {};
[-3,-1,-2]
\filldraw [draw=black,thick,fill opacity = .5, fill = Gainsboro,shift={(4.125000,-0.649519)}](0,0) -- (60:0.75cm) -- (0:0.75cm) --(0,0) node[blue] at (0.3cm,0.15cm) {} node[black] at (0.3cm,0.3cm) {} node[blue] at (0.3cm,0.45cm) {};
[-2,-1,-2]
\filldraw [draw=black, thick,fill = Gainsboro, shift={(4.875,-0.649519)}](0,0) -- (60:0.75cm) -- (120:0.75cm) --(0,0) node[green,fill opacity=1] at (0cm,0.6cm) {} node[black,fill opacity=1, draw opacity = 1] at (0cm,0.45cm) {} node[green,fill opacity=1] at (0cm,0.3cm) {};
[-2,-1,-1]
\filldraw [draw=black,thick,fill opacity = .5, fill = Goldenrod,shift={(4.875000,-0.649519)}](0,0) -- (60:0.75cm) -- (0:0.75cm) --(0,0) node[blue] at (0.3cm,0.15cm) {} node[black] at (0.3cm,0.3cm) {} node[blue] at (0.3cm,0.45cm) {};
\begin{pgfonlayer}{foreground layer}
\fill[fill=Goldenrod, fill opacity=1.](4.875,-0.649519) circle (0.0675cm);
\draw[black, fill opacity=1.](4.875,-0.649519) circle (0.0825cm);
\end{pgfonlayer}{foreground layer}
[-1,-1,-1]
\filldraw [draw=black, thick,fill = Gainsboro, shift={(5.625,-0.649519)}](0,0) -- (60:0.75cm) -- (120:0.75cm) --(0,0) node[green,fill opacity=1] at (0cm,0.6cm) {} node[black,fill opacity=1, draw opacity = 1] at (0cm,0.45cm) {} node[green,fill opacity=1] at (0cm,0.3cm) {};
[-1,-1,0]
\filldraw [draw=black,thick,fill opacity = .5, fill = Gainsboro,shift={(5.625000,-0.649519)}](0,0) -- (60:0.75cm) -- (0:0.75cm) --(0,0) node[blue] at (0.3cm,0.15cm) {} node[black] at (0.3cm,0.3cm) {} node[blue] at (0.3cm,0.45cm) {};
[0,-1,0]
\filldraw [draw=black, thick,fill = Gainsboro, shift={(6.375,-0.649519)}](0,0) -- (60:0.75cm) -- (120:0.75cm) --(0,0) node[green,fill opacity=1] at (0cm,0.6cm) {} node[black,fill opacity=1, draw opacity = 1] at (0cm,0.45cm) {} node[green,fill opacity=1] at (0cm,0.3cm) {};
[0,-1,1]
\filldraw [draw=black,thick,fill opacity = .5, fill = Goldenrod,shift={(6.375000,-0.649519)}](0,0) -- (60:0.75cm) -- (0:0.75cm) --(0,0) node[blue] at (0.3cm,0.15cm) {} node[black] at (0.3cm,0.3cm) {} node[blue] at (0.3cm,0.45cm) {};
\begin{pgfonlayer}{foreground layer}
\fill[fill=Goldenrod, fill opacity=1.](7.125,-0.649519) circle (0.0675cm);
\draw[black, fill opacity=1.](7.125,-0.649519) circle (0.0825cm);
\end{pgfonlayer}{foreground layer}
[1,-1,1]
\filldraw [draw=black, thick,fill = Gainsboro, shift={(7.125,-0.649519)}](0,0) -- (60:0.75cm) -- (120:0.75cm) --(0,0) node[green,fill opacity=1] at (0cm,0.6cm) {} node[black,fill opacity=1, draw opacity = 1] at (0cm,0.45cm) {} node[green,fill opacity=1] at (0cm,0.3cm) {};
[1,-1,2]
\filldraw [draw=black,thick,fill opacity = .5, fill = Gainsboro,shift={(7.125000,-0.649519)}](0,0) -- (60:0.75cm) -- (0:0.75cm) --(0,0) node[blue] at (0.3cm,0.15cm) {} node[black] at (0.3cm,0.3cm) {} node[blue] at (0.3cm,0.45cm) {};
[2,-1,2]
\filldraw [draw=black, thick,fill = Gainsboro, shift={(7.875,-0.649519)}](0,0) -- (60:0.75cm) -- (120:0.75cm) --(0,0) node[green,fill opacity=1] at (0cm,0.6cm) {} node[black,fill opacity=1, draw opacity = 1] at (0cm,0.45cm) {} node[green,fill opacity=1] at (0cm,0.3cm) {};
[2,-1,3]
\filldraw [draw=black,thick,fill opacity = .5, fill = Gainsboro,shift={(7.875000,-0.649519)}](0,0) -- (60:0.75cm) -- (0:0.75cm) --(0,0) node[blue] at (0.3cm,0.15cm) {} node[black] at (0.3cm,0.3cm) {} node[blue] at (0.3cm,0.45cm) {};
[-2,0,-2]
\filldraw [draw=black,thick,fill opacity = .5, fill = Gainsboro,shift={(4.500000,0)}](0,0) -- (60:0.75cm) -- (0:0.75cm) --(0,0) node[blue] at (0.3cm,0.15cm) {} node[black] at (0.3cm,0.3cm) {} node[blue] at (0.3cm,0.45cm) {};
[-1,0,-2]
\filldraw [draw=black, thick,fill = blue, shift={(5.25,0)}](0,0) -- (60:0.75cm) -- (120:0.75cm) --(0,0) node[green,fill opacity=1] at (0cm,0.6cm) {} node[black,fill opacity=1, draw opacity = 1] at (0cm,0.45cm) {} node[green,fill opacity=1] at (0cm,0.3cm) {};
\begin{pgfonlayer}{foreground layer}
\fill[fill=blue, fill opacity=1.](4.875,0.649519) circle (0.0675cm);
\draw[black, fill opacity=1.](4.875,0.649519) circle (0.0825cm);
\end{pgfonlayer}{foreground layer}
[-1,0,-1]
\filldraw [draw=black,thick,fill opacity = .5, fill = Gainsboro,shift={(5.250000,0)}](0,0) -- (60:0.75cm) -- (0:0.75cm) --(0,0) node[blue] at (0.3cm,0.15cm) {} node[black] at (0.3cm,0.3cm) {} node[blue] at (0.3cm,0.45cm) {};
[0,0,-1]
\filldraw [draw=black, thick,fill = Gainsboro, shift={(6,0)}](0,0) -- (60:0.75cm) -- (120:0.75cm) --(0,0) node[green,fill opacity=1] at (0cm,0.6cm) {} node[black,fill opacity=1, draw opacity = 1] at (0cm,0.45cm) {} node[green,fill opacity=1] at (0cm,0.3cm) {};
[0,0,0]
\filldraw [draw=black,thick,fill opacity = .5, fill = Goldenrod,shift={(6.000000,0)}](0,0) -- (60:0.75cm) -- (0:0.75cm) --(0,0) node[blue] at (0.3cm,0.15cm) {} node[black] at (0.3cm,0.3cm) {} node[blue] at (0.3cm,0.45cm) {};
\begin{pgfonlayer}{foreground layer}
\fill[fill=Goldenrod, fill opacity=1.](6,0) circle (0.0675cm);
\draw[black, fill opacity=1.](6,0) circle (0.0825cm);
\end{pgfonlayer}{foreground layer}
[1,0,0]
\filldraw [draw=black, thick,fill = Goldenrod, shift={(6.75,0)}](0,0) -- (60:0.75cm) -- (120:0.75cm) --(0,0) node[green,fill opacity=1] at (0cm,0.6cm) {} node[black,fill opacity=1, draw opacity = 1] at (0cm,0.45cm) {} node[green,fill opacity=1] at (0cm,0.3cm) {};
\begin{pgfonlayer}{foreground layer}
\fill[fill=Goldenrod, fill opacity=1.](7.125,0.649519) circle (0.0675cm);
\draw[black, fill opacity=1.](7.125,0.649519) circle (0.0825cm);
\end{pgfonlayer}{foreground layer}
[1,0,1]
\filldraw [draw=black,thick,fill opacity = .5, fill = Gainsboro,shift={(6.750000,0)}](0,0) -- (60:0.75cm) -- (0:0.75cm) --(0,0) node[blue] at (0.3cm,0.15cm) {} node[black] at (0.3cm,0.3cm) {} node[blue] at (0.3cm,0.45cm) {};
[2,0,1]
\filldraw [draw=black, thick,fill = Gainsboro, shift={(7.5,0)}](0,0) -- (60:0.75cm) -- (120:0.75cm) --(0,0) node[green,fill opacity=1] at (0cm,0.6cm) {} node[black,fill opacity=1, draw opacity = 1] at (0cm,0.45cm) {} node[green,fill opacity=1] at (0cm,0.3cm) {};
[2,0,2]
\filldraw [draw=black,thick,fill opacity = .5, fill = blue,shift={(7.500000,0)}](0,0) -- (60:0.75cm) -- (0:0.75cm) --(0,0) node[blue] at (0.3cm,0.15cm) {} node[black] at (0.3cm,0.3cm) {} node[blue] at (0.3cm,0.45cm) {};
\begin{pgfonlayer}{foreground layer}
\fill[fill=blue, fill opacity=1.](8.25,0) circle (0.0675cm);
\draw[black, fill opacity=1.](8.25,0) circle (0.0825cm);
\end{pgfonlayer}{foreground layer}
[-1,1,-2]
\filldraw [draw=black,thick,fill opacity = .5, fill = Gainsboro,shift={(4.875000,0.649519)}](0,0) -- (60:0.75cm) -- (0:0.75cm) --(0,0) node[blue] at (0.3cm,0.15cm) {} node[black] at (0.3cm,0.3cm) {} node[blue] at (0.3cm,0.45cm) {};
[0,1,-2]
\filldraw [draw=black, thick,fill = Gainsboro, shift={(5.625,0.649519)}](0,0) -- (60:0.75cm) -- (120:0.75cm) --(0,0) node[green,fill opacity=1] at (0cm,0.6cm) {} node[black,fill opacity=1, draw opacity = 1] at (0cm,0.45cm) {} node[green,fill opacity=1] at (0cm,0.3cm) {};
[0,1,-1]
\filldraw [draw=black,thick,fill opacity = .5, fill = Goldenrod,shift={(5.625000,0.649519)}](0,0) -- (60:0.75cm) -- (0:0.75cm) --(0,0) node[blue] at (0.3cm,0.15cm) {} node[black] at (0.3cm,0.3cm) {} node[blue] at (0.3cm,0.45cm) {};
\begin{pgfonlayer}{foreground layer}
\fill[fill=Goldenrod, fill opacity=1.](6,1.29904) circle (0.0675cm);
\draw[black, fill opacity=1.](6,1.29904) circle (0.0825cm);
\end{pgfonlayer}{foreground layer}
[1,1,-1]
\filldraw [draw=black, thick,fill = Gainsboro, shift={(6.375,0.649519)}](0,0) -- (60:0.75cm) -- (120:0.75cm) --(0,0) node[green,fill opacity=1] at (0cm,0.6cm) {} node[black,fill opacity=1, draw opacity = 1] at (0cm,0.45cm) {} node[green,fill opacity=1] at (0cm,0.3cm) {};
[1,1,0]
\filldraw [draw=black,thick,fill opacity = .5, fill = Gainsboro,shift={(6.375000,0.649519)}](0,0) -- (60:0.75cm) -- (0:0.75cm) --(0,0) node[blue] at (0.3cm,0.15cm) {} node[black] at (0.3cm,0.3cm) {} node[blue] at (0.3cm,0.45cm) {};
[2,1,0]
\filldraw [draw=black, thick,fill = Gainsboro, shift={(7.125,0.649519)}](0,0) -- (60:0.75cm) -- (120:0.75cm) --(0,0) node[green,fill opacity=1] at (0cm,0.6cm) {} node[black,fill opacity=1, draw opacity = 1] at (0cm,0.45cm) {} node[green,fill opacity=1] at (0cm,0.3cm) {};
[2,1,1]
\filldraw [draw=black,thick,fill opacity = .5, fill = Gainsboro,shift={(7.125000,0.649519)}](0,0) -- (60:0.75cm) -- (0:0.75cm) --(0,0) node[blue] at (0.3cm,0.15cm) {} node[black] at (0.3cm,0.3cm) {} node[blue] at (0.3cm,0.45cm) {};
[0,2,-2]
\filldraw [draw=black,thick,fill opacity = .5, fill = Gainsboro,shift={(5.250000,1.29904)}](0,0) -- (60:0.75cm) -- (0:0.75cm) --(0,0) node[blue] at (0.3cm,0.15cm) {} node[black] at (0.3cm,0.3cm) {} node[blue] at (0.3cm,0.45cm) {};
[1,2,-2]
\filldraw [draw=black, thick,fill = Gainsboro, shift={(6,1.29904)}](0,0) -- (60:0.75cm) -- (120:0.75cm) --(0,0) node[green,fill opacity=1] at (0cm,0.6cm) {} node[black,fill opacity=1, draw opacity = 1] at (0cm,0.45cm) {} node[green,fill opacity=1] at (0cm,0.3cm) {};
[1,2,-1]
\filldraw [draw=black,thick,fill opacity = .5, fill = Gainsboro,shift={(6.000000,1.29904)}](0,0) -- (60:0.75cm) -- (0:0.75cm) --(0,0) node[blue] at (0.3cm,0.15cm) {} node[black] at (0.3cm,0.3cm) {} node[blue] at (0.3cm,0.45cm) {};
[2,2,-1]
\filldraw [draw=black, thick,fill = Gainsboro, shift={(6.75,1.29904)}](0,0) -- (60:0.75cm) -- (120:0.75cm) --(0,0) node[green,fill opacity=1] at (0cm,0.6cm) {} node[black,fill opacity=1, draw opacity = 1] at (0cm,0.45cm) {} node[green,fill opacity=1] at (0cm,0.3cm) {};
[2,2,0]
\filldraw [draw=black,thick,fill opacity = .5, fill = blue,shift={(6.750000,1.29904)}](0,0) -- (60:0.75cm) -- (0:0.75cm) --(0,0) node[blue] at (0.3cm,0.15cm) {} node[black] at (0.3cm,0.3cm) {} node[blue] at (0.3cm,0.45cm) {};
\begin{pgfonlayer}{foreground layer}
\fill[fill=blue, fill opacity=1.](7.125,1.94856) circle (0.0675cm);
\draw[black, fill opacity=1.](7.125,1.94856) circle (0.0825cm);
\end{pgfonlayer}{foreground layer}
[1,3,-2]
\filldraw [draw=black,thick,fill opacity = .5, fill = blue,shift={(5.625000,1.94856)}](0,0) -- (60:0.75cm) -- (0:0.75cm) --(0,0) node[blue] at (0.3cm,0.15cm) {} node[black] at (0.3cm,0.3cm) {} node[blue] at (0.3cm,0.45cm) {};
\begin{pgfonlayer}{foreground layer}
\fill[fill=blue, fill opacity=1.](6,2.59808) circle (0.0675cm);
\draw[black, fill opacity=1.](6,2.59808) circle (0.0825cm);
\end{pgfonlayer}{foreground layer}
[2,3,-2]
\filldraw [draw=black, thick,fill = Gainsboro, shift={(6.375,1.94856)}](0,0) -- (60:0.75cm) -- (120:0.75cm) --(0,0) node[green,fill opacity=1] at (0cm,0.6cm) {} node[black,fill opacity=1, draw opacity = 1] at (0cm,0.45cm) {} node[green,fill opacity=1] at (0cm,0.3cm) {};
[2,3,-1]
\filldraw [draw=black,thick,fill opacity = .5, fill = Gainsboro,shift={(6.375000,1.94856)}](0,0) -- (60:0.75cm) -- (0:0.75cm) --(0,0) node[blue] at (0.3cm,0.15cm) {} node[black] at (0.3cm,0.3cm) {} node[blue] at (0.3cm,0.45cm) {};
[2,4,-2]
\filldraw [draw=black,thick,fill opacity = .5, fill = Gainsboro,shift={(6.000000,2.59808)}](0,0) -- (60:0.75cm) -- (0:0.75cm) --(0,0) node[blue] at (0.3cm,0.15cm) {} node[black] at (0.3cm,0.3cm) {} node[blue] at (0.3cm,0.45cm) {};
\path [draw = red, very thick,  shift = {(5.25,0)}](-1.5,-2.59808) -- (1.5,2.59808);
\path [draw = red,very thick, shift = {(6,-0.649519)}](-3,0) -- (3,0);
\path [draw = red, very thick,  shift = {(5.25,0)}](-1.5,2.59808) -- (1.5,-2.59808);
\path [draw = red, very thick,  shift = {(6,0)}](-1.5,-2.59808) -- (1.5,2.59808);
\path [draw = red,very thick, shift = {(6,0)}](-3,0) -- (3,0);
\path [draw = red, very thick,  shift = {(6,0)}](-1.5,2.59808) -- (1.5,-2.59808);
\path [draw = red, very thick,  shift = {(6.75,0)}](-1.5,-2.59808) -- (1.5,2.59808);
\path [draw = red,very thick, shift = {(6,0.649519)}](-3,0) -- (3,0);
\path [draw = red, very thick,  shift = {(6.75,0)}](-1.5,2.59808) -- (1.5,-2.59808);
\path [draw = red, very thick,  shift = {(7.5,0)}](-1.5,-2.59808) -- (1.5,2.59808);
\path [draw = red,very thick, shift = {(6,1.29904)}](-3,0) -- (3,0);
\path [draw = red, very thick,  shift = {(7.5,0)}](-1.5,2.59808) -- (1.5,-2.59808);
\fill[fill=black, fill opacity=1.](6,0.) circle (0.0675cm);\end{tikzpicture}
  \caption{\small 
$w \Afund$ for the $\m$-minimal alcoves $w^{-1} \Afund$ in Figure
\ref{fig;onetwoshi}
below, $m=1,2$.  Note $w \Afund \subseteq \Region$. 
Each $\gamma \in Q$ is in precisely one yellow/blue alcove,
so this  illustrates the second statement of Theorem \ref{thm;haiman}. }
  \label{fig;dilate}
\end{figure}
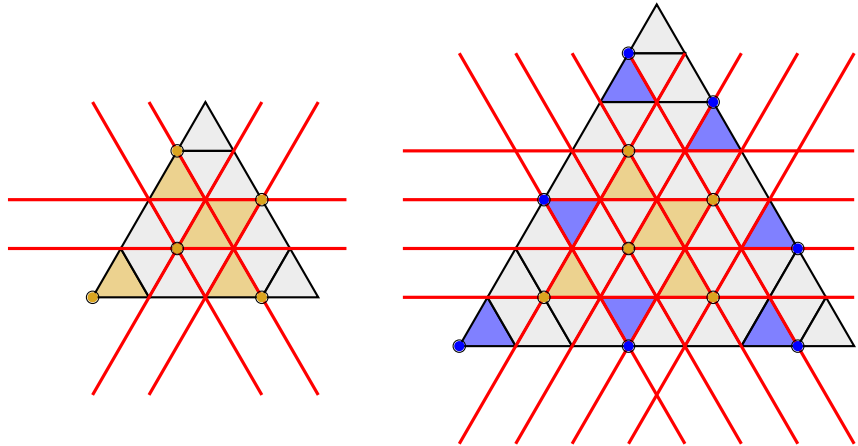

\begin{figure}[ht]

\begin{tikzpicture}[fill opacity=.5]
\tikzstyle{every node}=[font=\footnotesize]
[-1,-2,1]
\filldraw [draw=black,thick,fill opacity = .5, fill = Gainsboro,shift={(0.000000,-1.29904)}](0,0) -- (60:0.75cm) -- (0:0.75cm) --(0,0) node[blue] at (0.3cm,0.15cm) {} node[black] at (0.3cm,0.3cm) {} node[blue] at (0.3cm,0.45cm) {};
[-1,-1,-1]
\filldraw [draw=black, thick,fill = Gainsboro, shift={(-0.375,-0.649519)}](0,0) -- (60:0.75cm) -- (120:0.75cm) --(0,0) node[green,fill opacity=1] at (0cm,0.6cm) {} node[black,fill opacity=1, draw opacity = 1] at (0cm,0.45cm) {} node[green,fill opacity=1] at (0cm,0.3cm) {};
[-1,-1,0]
\filldraw [draw=black,thick,fill opacity = .5, fill = Gainsboro,shift={(-0.375000,-0.649519)}](0,0) -- (60:0.75cm) -- (0:0.75cm) --(0,0) node[blue] at (0.3cm,0.15cm) {} node[black] at (0.3cm,0.3cm) {} node[blue] at (0.3cm,0.45cm) {};
[0,-1,0]
\filldraw [draw=black, thick,fill = Gainsboro, shift={(0.375,-0.649519)}](0,0) -- (60:0.75cm) -- (120:0.75cm) --(0,0) node[green,fill opacity=1] at (0cm,0.6cm) {} node[black,fill opacity=1, draw opacity = 1] at (0cm,0.45cm) {} node[green,fill opacity=1] at (0cm,0.3cm) {};
[0,-1,1]
\filldraw [draw=black,thick,fill opacity = .5, fill = Gainsboro,shift={(0.375000,-0.649519)}](0,0) -- (60:0.75cm) -- (0:0.75cm) --(0,0) node[blue] at (0.3cm,0.15cm) {} node[black] at (0.3cm,0.3cm) {} node[blue] at (0.3cm,0.45cm) {};
[1,-1,1]
\filldraw [draw=black, thick,fill = Gainsboro, shift={(1.125,-0.649519)}](0,0) -- (60:0.75cm) -- (120:0.75cm) --(0,0) node[green,fill opacity=1] at (0cm,0.6cm) {} node[black,fill opacity=1, draw opacity = 1] at (0cm,0.45cm) {} node[green,fill opacity=1] at (0cm,0.3cm) {};
[-1,0,-1]
\filldraw [draw=black,thick,fill opacity = .5, fill = Gainsboro,shift={(-0.750000,0)}](0,0) -- (60:0.75cm) -- (0:0.75cm) --(0,0) node[blue] at (0.3cm,0.15cm) {} node[black] at (0.3cm,0.3cm) {} node[blue] at (0.3cm,0.45cm) {};
[0,0,-1]
\filldraw [draw=black, thick,fill = Gainsboro, shift={(0,0)}](0,0) -- (60:0.75cm) -- (120:0.75cm) --(0,0) node[green,fill opacity=1] at (0cm,0.6cm) {} node[black,fill opacity=1, draw opacity = 1] at (0cm,0.45cm) {} node[green,fill opacity=1] at (0cm,0.3cm) {};
[0,0,0]
\filldraw [draw=black,thick,fill opacity = .5, fill = Goldenrod,shift={(0.000000,0)}](0,0) -- (60:0.75cm) -- (0:0.75cm) --(0,0) node[blue] at (0.3cm,0.15cm) {} node[black] at (0.3cm,0.3cm) {} node[blue] at (0.3cm,0.45cm) {};
[1,0,0]
\filldraw [draw=black, thick,fill = Goldenrod, shift={(0.75,0)}](0,0) -- (60:0.75cm) -- (120:0.75cm) --(0,0) node[green,fill opacity=1] at (0cm,0.6cm) {} node[black,fill opacity=1, draw opacity = 1] at (0cm,0.45cm) {} node[green,fill opacity=1] at (0cm,0.3cm) {};
[1,0,1]
\filldraw [draw=black,thick,fill opacity = .5, fill = Goldenrod,shift={(0.750000,0)}](0,0) -- (60:0.75cm) -- (0:0.75cm) --(0,0) node[blue] at (0.3cm,0.15cm) {} node[black] at (0.3cm,0.3cm) {} node[blue] at (0.3cm,0.45cm) {};
[-1,1,-2]
\filldraw [draw=black,thick,fill opacity = .5, fill = Gainsboro,shift={(-1.125000,0.649519)}](0,0) -- (60:0.75cm) -- (0:0.75cm) --(0,0) node[blue] at (0.3cm,0.15cm) {} node[black] at (0.3cm,0.3cm) {} node[blue] at (0.3cm,0.45cm) {};
[0,1,-1]
\filldraw [draw=black,thick,fill opacity = .5, fill = Gainsboro,shift={(-0.375000,0.649519)}](0,0) -- (60:0.75cm) -- (0:0.75cm) --(0,0) node[blue] at (0.3cm,0.15cm) {} node[black] at (0.3cm,0.3cm) {} node[blue] at (0.3cm,0.45cm) {};
[1,1,-1]
\filldraw [draw=black, thick,fill = Gainsboro, shift={(0.375,0.649519)}](0,0) -- (60:0.75cm) -- (120:0.75cm) --(0,0) node[green,fill opacity=1] at (0cm,0.6cm) {} node[black,fill opacity=1, draw opacity = 1] at (0cm,0.45cm) {} node[green,fill opacity=1] at (0cm,0.3cm) {};
[1,1,0]
\filldraw [draw=black,thick,fill opacity = .5, fill = Goldenrod,shift={(0.375000,0.649519)}](0,0) -- (60:0.75cm) -- (0:0.75cm) --(0,0) node[blue] at (0.3cm,0.15cm) {} node[black] at (0.3cm,0.3cm) {} node[blue] at (0.3cm,0.45cm) {};
[2,1,1]
\filldraw [draw=black,thick,fill opacity = .5, fill = Goldenrod,shift={(1.125000,0.649519)}](0,0) -- (60:0.75cm) -- (0:0.75cm) --(0,0) node[blue] at (0.3cm,0.15cm) {} node[black] at (0.3cm,0.3cm) {} node[blue] at (0.3cm,0.45cm) {};
\path [draw = red, very thick,  shift = {(0,0)}](-1.125,-1.94856) -- (1.125,1.94856);
\path [draw = red,very thick, shift = {(0,0)}](-2.25,0) -- (2.25,0);
\path [draw = red, very thick,  shift = {(0,0)}](-1.125,1.94856) -- (1.125,-1.94856);
\path [draw = red, very thick,  shift = {(0.75,0)}](-1.125,-1.94856) -- (1.125,1.94856);
\path [draw = red,very thick, shift = {(0,0.649519)}](-2.25,0) -- (2.25,0);
\path [draw = red, very thick,  shift = {(0.75,0)}](-1.125,1.94856) -- (1.125,-1.94856);
\fill[fill=black, fill opacity=1.](0.,0.) circle (0.0675cm);[-2,-4,2]
\filldraw [draw=black,thick,fill opacity = .5, fill = Gainsboro,shift={(6.000000,-2.59808)}](0,0) -- (60:0.75cm) -- (0:0.75cm) --(0,0) node[blue] at (0.3cm,0.15cm) {} node[black] at (0.3cm,0.3cm) {} node[blue] at (0.3cm,0.45cm) {};
[-2,-3,1]
\filldraw [draw=black,thick,fill opacity = .5, fill = Gainsboro,shift={(5.625000,-1.94856)}](0,0) -- (60:0.75cm) -- (0:0.75cm) --(0,0) node[blue] at (0.3cm,0.15cm) {} node[black] at (0.3cm,0.3cm) {} node[blue] at (0.3cm,0.45cm) {};
[-1,-3,2]
\filldraw [draw=black,thick,fill opacity = .5, fill = Gainsboro,shift={(6.375000,-1.94856)}](0,0) -- (60:0.75cm) -- (0:0.75cm) --(0,0) node[blue] at (0.3cm,0.15cm) {} node[black] at (0.3cm,0.3cm) {} node[blue] at (0.3cm,0.45cm) {};
[-3,-2,-2]
\filldraw [draw=black, thick,fill = Gainsboro, shift={(4.5,-1.29904)}](0,0) -- (60:0.75cm) -- (120:0.75cm) --(0,0) node[green,fill opacity=1] at (0cm,0.6cm) {} node[black,fill opacity=1, draw opacity = 1] at (0cm,0.45cm) {} node[green,fill opacity=1] at (0cm,0.3cm) {};
[-2,-2,-1]
\filldraw [draw=black, thick,fill = Gainsboro, shift={(5.25,-1.29904)}](0,0) -- (60:0.75cm) -- (120:0.75cm) --(0,0) node[green,fill opacity=1] at (0cm,0.6cm) {} node[black,fill opacity=1, draw opacity = 1] at (0cm,0.45cm) {} node[green,fill opacity=1] at (0cm,0.3cm) {};
[-2,-2,0]
\filldraw [draw=black,thick,fill opacity = .5, fill = Gainsboro,shift={(5.250000,-1.29904)}](0,0) -- (60:0.75cm) -- (0:0.75cm) --(0,0) node[blue] at (0.3cm,0.15cm) {} node[black] at (0.3cm,0.3cm) {} node[blue] at (0.3cm,0.45cm) {};
[-1,-2,0]
\filldraw [draw=black, thick,fill = Gainsboro, shift={(6,-1.29904)}](0,0) -- (60:0.75cm) -- (120:0.75cm) --(0,0) node[green,fill opacity=1] at (0cm,0.6cm) {} node[black,fill opacity=1, draw opacity = 1] at (0cm,0.45cm) {} node[green,fill opacity=1] at (0cm,0.3cm) {};
[-1,-2,1]
\filldraw [draw=black,thick,fill opacity = .5, fill = Gainsboro,shift={(6.000000,-1.29904)}](0,0) -- (60:0.75cm) -- (0:0.75cm) --(0,0) node[blue] at (0.3cm,0.15cm) {} node[black] at (0.3cm,0.3cm) {} node[blue] at (0.3cm,0.45cm) {};
[0,-2,1]
\filldraw [draw=black, thick,fill = Gainsboro, shift={(6.75,-1.29904)}](0,0) -- (60:0.75cm) -- (120:0.75cm) --(0,0) node[green,fill opacity=1] at (0cm,0.6cm) {} node[black,fill opacity=1, draw opacity = 1] at (0cm,0.45cm) {} node[green,fill opacity=1] at (0cm,0.3cm) {};
[0,-2,2]
\filldraw [draw=black,thick,fill opacity = .5, fill = Gainsboro,shift={(6.750000,-1.29904)}](0,0) -- (60:0.75cm) -- (0:0.75cm) --(0,0) node[blue] at (0.3cm,0.15cm) {} node[black] at (0.3cm,0.3cm) {} node[blue] at (0.3cm,0.45cm) {};
[1,-2,2]
\filldraw [draw=black, thick,fill = Gainsboro, shift={(7.5,-1.29904)}](0,0) -- (60:0.75cm) -- (120:0.75cm) --(0,0) node[green,fill opacity=1] at (0cm,0.6cm) {} node[black,fill opacity=1, draw opacity = 1] at (0cm,0.45cm) {} node[green,fill opacity=1] at (0cm,0.3cm) {};
[2,-2,3]
\filldraw [draw=black, thick,fill = Gainsboro, shift={(8.25,-1.29904)}](0,0) -- (60:0.75cm) -- (120:0.75cm) --(0,0) node[green,fill opacity=1] at (0cm,0.6cm) {} node[black,fill opacity=1, draw opacity = 1] at (0cm,0.45cm) {} node[green,fill opacity=1] at (0cm,0.3cm) {};
[-2,-1,-2]
\filldraw [draw=black, thick,fill = Gainsboro, shift={(4.875,-0.649519)}](0,0) -- (60:0.75cm) -- (120:0.75cm) --(0,0) node[green,fill opacity=1] at (0cm,0.6cm) {} node[black,fill opacity=1, draw opacity = 1] at (0cm,0.45cm) {} node[green,fill opacity=1] at (0cm,0.3cm) {};
[-2,-1,-1]
\filldraw [draw=black,thick,fill opacity = .5, fill = Gainsboro,shift={(4.875000,-0.649519)}](0,0) -- (60:0.75cm) -- (0:0.75cm) --(0,0) node[blue] at (0.3cm,0.15cm) {} node[black] at (0.3cm,0.3cm) {} node[blue] at (0.3cm,0.45cm) {};
[-1,-1,-1]
\filldraw [draw=black, thick,fill = Gainsboro, shift={(5.625,-0.649519)}](0,0) -- (60:0.75cm) -- (120:0.75cm) --(0,0) node[green,fill opacity=1] at (0cm,0.6cm) {} node[black,fill opacity=1, draw opacity = 1] at (0cm,0.45cm) {} node[green,fill opacity=1] at (0cm,0.3cm) {};
[-1,-1,0]
\filldraw [draw=black,thick,fill opacity = .5, fill = Gainsboro,shift={(5.625000,-0.649519)}](0,0) -- (60:0.75cm) -- (0:0.75cm) --(0,0) node[blue] at (0.3cm,0.15cm) {} node[black] at (0.3cm,0.3cm) {} node[blue] at (0.3cm,0.45cm) {};
[0,-1,0]
\filldraw [draw=black, thick,fill = Gainsboro, shift={(6.375,-0.649519)}](0,0) -- (60:0.75cm) -- (120:0.75cm) --(0,0) node[green,fill opacity=1] at (0cm,0.6cm) {} node[black,fill opacity=1, draw opacity = 1] at (0cm,0.45cm) {} node[green,fill opacity=1] at (0cm,0.3cm) {};
[0,-1,1]
\filldraw [draw=black,thick,fill opacity = .5, fill = Gainsboro,shift={(6.375000,-0.649519)}](0,0) -- (60:0.75cm) -- (0:0.75cm) --(0,0) node[blue] at (0.3cm,0.15cm) {} node[black] at (0.3cm,0.3cm) {} node[blue] at (0.3cm,0.45cm) {};
[1,-1,1]
\filldraw [draw=black, thick,fill = Gainsboro, shift={(7.125,-0.649519)}](0,0) -- (60:0.75cm) -- (120:0.75cm) --(0,0) node[green,fill opacity=1] at (0cm,0.6cm) {} node[black,fill opacity=1, draw opacity = 1] at (0cm,0.45cm) {} node[green,fill opacity=1] at (0cm,0.3cm) {};
[1,-1,2]
\filldraw [draw=black,thick,fill opacity = .5, fill = Gainsboro,shift={(7.125000,-0.649519)}](0,0) -- (60:0.75cm) -- (0:0.75cm) --(0,0) node[blue] at (0.3cm,0.15cm) {} node[black] at (0.3cm,0.3cm) {} node[blue] at (0.3cm,0.45cm) {};
[2,-1,2]
\filldraw [draw=black, thick,fill = Gainsboro, shift={(7.875,-0.649519)}](0,0) -- (60:0.75cm) -- (120:0.75cm) --(0,0) node[green,fill opacity=1] at (0cm,0.6cm) {} node[black,fill opacity=1, draw opacity = 1] at (0cm,0.45cm) {} node[green,fill opacity=1] at (0cm,0.3cm) {};
[-2,0,-2]
\filldraw [draw=black,thick,fill opacity = .5, fill = Gainsboro,shift={(4.500000,0)}](0,0) -- (60:0.75cm) -- (0:0.75cm) --(0,0) node[blue] at (0.3cm,0.15cm) {} node[black] at (0.3cm,0.3cm) {} node[blue] at (0.3cm,0.45cm) {};
[-1,0,-2]
\filldraw [draw=black, thick,fill = Gainsboro, shift={(5.25,0)}](0,0) -- (60:0.75cm) -- (120:0.75cm) --(0,0) node[green,fill opacity=1] at (0cm,0.6cm) {} node[black,fill opacity=1, draw opacity = 1] at (0cm,0.45cm) {} node[green,fill opacity=1] at (0cm,0.3cm) {};
[-1,0,-1]
\filldraw [draw=black,thick,fill opacity = .5, fill = Gainsboro,shift={(5.250000,0)}](0,0) -- (60:0.75cm) -- (0:0.75cm) --(0,0) node[blue] at (0.3cm,0.15cm) {} node[black] at (0.3cm,0.3cm) {} node[blue] at (0.3cm,0.45cm) {};
[0,0,-1]
\filldraw [draw=black, thick,fill = Gainsboro, shift={(6,0)}](0,0) -- (60:0.75cm) -- (120:0.75cm) --(0,0) node[green,fill opacity=1] at (0cm,0.6cm) {} node[black,fill opacity=1, draw opacity = 1] at (0cm,0.45cm) {} node[green,fill opacity=1] at (0cm,0.3cm) {};
[0,0,0]
\filldraw [draw=black,thick,fill opacity = .5, fill = Goldenrod,shift={(6.000000,0)}](0,0) -- (60:0.75cm) -- (0:0.75cm) --(0,0) node[blue] at (0.3cm,0.15cm) {} node[black] at (0.3cm,0.3cm) {} node[blue] at (0.3cm,0.45cm) {};
[1,0,0]
\filldraw [draw=black, thick,fill = Goldenrod, shift={(6.75,0)}](0,0) -- (60:0.75cm) -- (120:0.75cm) --(0,0) node[green,fill opacity=1] at (0cm,0.6cm) {} node[black,fill opacity=1, draw opacity = 1] at (0cm,0.45cm) {} node[green,fill opacity=1] at (0cm,0.3cm) {};
[1,0,1]
\filldraw [draw=black,thick,fill opacity = .5, fill = Goldenrod,shift={(6.750000,0)}](0,0) -- (60:0.75cm) -- (0:0.75cm) --(0,0) node[blue] at (0.3cm,0.15cm) {} node[black] at (0.3cm,0.3cm) {} node[blue] at (0.3cm,0.45cm) {};
[2,0,1]
\filldraw [draw=black, thick,fill = blue, shift={(7.5,0)}](0,0) -- (60:0.75cm) -- (120:0.75cm) --(0,0) node[green,fill opacity=1] at (0cm,0.6cm) {} node[black,fill opacity=1, draw opacity = 1] at (0cm,0.45cm) {} node[green,fill opacity=1] at (0cm,0.3cm) {};
[2,0,2]
\filldraw [draw=black,thick,fill opacity = .5, fill = blue,shift={(7.500000,0)}](0,0) -- (60:0.75cm) -- (0:0.75cm) --(0,0) node[blue] at (0.3cm,0.15cm) {} node[black] at (0.3cm,0.3cm) {} node[blue] at (0.3cm,0.45cm) {};
[-2,1,-3]
\filldraw [draw=black,thick,fill opacity = .5, fill = Gainsboro,shift={(4.125000,0.649519)}](0,0) -- (60:0.75cm) -- (0:0.75cm) --(0,0) node[blue] at (0.3cm,0.15cm) {} node[black] at (0.3cm,0.3cm) {} node[blue] at (0.3cm,0.45cm) {};
[-1,1,-2]
\filldraw [draw=black,thick,fill opacity = .5, fill = Gainsboro,shift={(4.875000,0.649519)}](0,0) -- (60:0.75cm) -- (0:0.75cm) --(0,0) node[blue] at (0.3cm,0.15cm) {} node[black] at (0.3cm,0.3cm) {} node[blue] at (0.3cm,0.45cm) {};
[0,1,-2]
\filldraw [draw=black, thick,fill = Gainsboro, shift={(5.625,0.649519)}](0,0) -- (60:0.75cm) -- (120:0.75cm) --(0,0) node[green,fill opacity=1] at (0cm,0.6cm) {} node[black,fill opacity=1, draw opacity = 1] at (0cm,0.45cm) {} node[green,fill opacity=1] at (0cm,0.3cm) {};
[0,1,-1]
\filldraw [draw=black,thick,fill opacity = .5, fill = Gainsboro,shift={(5.625000,0.649519)}](0,0) -- (60:0.75cm) -- (0:0.75cm) --(0,0) node[blue] at (0.3cm,0.15cm) {} node[black] at (0.3cm,0.3cm) {} node[blue] at (0.3cm,0.45cm) {};
[1,1,-1]
\filldraw [draw=black, thick,fill = Gainsboro, shift={(6.375,0.649519)}](0,0) -- (60:0.75cm) -- (120:0.75cm) --(0,0) node[green,fill opacity=1] at (0cm,0.6cm) {} node[black,fill opacity=1, draw opacity = 1] at (0cm,0.45cm) {} node[green,fill opacity=1] at (0cm,0.3cm) {};
[1,1,0]
\filldraw [draw=black,thick,fill opacity = .5, fill = Goldenrod,shift={(6.375000,0.649519)}](0,0) -- (60:0.75cm) -- (0:0.75cm) --(0,0) node[blue] at (0.3cm,0.15cm) {} node[black] at (0.3cm,0.3cm) {} node[blue] at (0.3cm,0.45cm) {};
[2,1,0]
\filldraw [draw=black, thick,fill = blue, shift={(7.125,0.649519)}](0,0) -- (60:0.75cm) -- (120:0.75cm) --(0,0) node[green,fill opacity=1] at (0cm,0.6cm) {} node[black,fill opacity=1, draw opacity = 1] at (0cm,0.45cm) {} node[green,fill opacity=1] at (0cm,0.3cm) {};
[2,1,1]
\filldraw [draw=black,thick,fill opacity = .5, fill = Goldenrod,shift={(7.125000,0.649519)}](0,0) -- (60:0.75cm) -- (0:0.75cm) --(0,0) node[blue] at (0.3cm,0.15cm) {} node[black] at (0.3cm,0.3cm) {} node[blue] at (0.3cm,0.45cm) {};
[3,1,2]
\filldraw [draw=black,thick,fill opacity = .5, fill = blue,shift={(7.875000,0.649519)}](0,0) -- (60:0.75cm) -- (0:0.75cm) --(0,0) node[blue] at (0.3cm,0.15cm) {} node[black] at (0.3cm,0.3cm) {} node[blue] at (0.3cm,0.45cm) {};
[-2,2,-4]
\filldraw [draw=black,thick,fill opacity = .5, fill = Gainsboro,shift={(3.750000,1.29904)}](0,0) -- (60:0.75cm) -- (0:0.75cm) --(0,0) node[blue] at (0.3cm,0.15cm) {} node[black] at (0.3cm,0.3cm) {} node[blue] at (0.3cm,0.45cm) {};
[-1,2,-3]
\filldraw [draw=black,thick,fill opacity = .5, fill = Gainsboro,shift={(4.500000,1.29904)}](0,0) -- (60:0.75cm) -- (0:0.75cm) --(0,0) node[blue] at (0.3cm,0.15cm) {} node[black] at (0.3cm,0.3cm) {} node[blue] at (0.3cm,0.45cm) {};
[0,2,-2]
\filldraw [draw=black,thick,fill opacity = .5, fill = Gainsboro,shift={(5.250000,1.29904)}](0,0) -- (60:0.75cm) -- (0:0.75cm) --(0,0) node[blue] at (0.3cm,0.15cm) {} node[black] at (0.3cm,0.3cm) {} node[blue] at (0.3cm,0.45cm) {};
[1,2,-2]
\filldraw [draw=black, thick,fill = Gainsboro, shift={(6,1.29904)}](0,0) -- (60:0.75cm) -- (120:0.75cm) --(0,0) node[green,fill opacity=1] at (0cm,0.6cm) {} node[black,fill opacity=1, draw opacity = 1] at (0cm,0.45cm) {} node[green,fill opacity=1] at (0cm,0.3cm) {};
[1,2,-1]
\filldraw [draw=black,thick,fill opacity = .5, fill = Gainsboro,shift={(6.000000,1.29904)}](0,0) -- (60:0.75cm) -- (0:0.75cm) --(0,0) node[blue] at (0.3cm,0.15cm) {} node[black] at (0.3cm,0.3cm) {} node[blue] at (0.3cm,0.45cm) {};
[2,2,-1]
\filldraw [draw=black, thick,fill = Gainsboro, shift={(6.75,1.29904)}](0,0) -- (60:0.75cm) -- (120:0.75cm) --(0,0) node[green,fill opacity=1] at (0cm,0.6cm) {} node[black,fill opacity=1, draw opacity = 1] at (0cm,0.45cm) {} node[green,fill opacity=1] at (0cm,0.3cm) {};
[2,2,0]
\filldraw [draw=black,thick,fill opacity = .5, fill = blue,shift={(6.750000,1.29904)}](0,0) -- (60:0.75cm) -- (0:0.75cm) --(0,0) node[blue] at (0.3cm,0.15cm) {} node[black] at (0.3cm,0.3cm) {} node[blue] at (0.3cm,0.45cm) {};
[3,2,1]
\filldraw [draw=black,thick,fill opacity = .5, fill = blue,shift={(7.500000,1.29904)}](0,0) -- (60:0.75cm) -- (0:0.75cm) --(0,0) node[blue] at (0.3cm,0.15cm) {} node[black] at (0.3cm,0.3cm) {} node[blue] at (0.3cm,0.45cm) {};
[4,2,2]
\filldraw [draw=black,thick,fill opacity = .5, fill = blue,shift={(8.250000,1.29904)}](0,0) -- (60:0.75cm) -- (0:0.75cm) --(0,0) node[blue] at (0.3cm,0.15cm) {} node[black] at (0.3cm,0.3cm) {} node[blue] at (0.3cm,0.45cm) {};
[2,3,-2]
\filldraw [draw=black, thick,fill = Gainsboro, shift={(6.375,1.94856)}](0,0) -- (60:0.75cm) -- (120:0.75cm) --(0,0) node[green,fill opacity=1] at (0cm,0.6cm) {} node[black,fill opacity=1, draw opacity = 1] at (0cm,0.45cm) {} node[green,fill opacity=1] at (0cm,0.3cm) {};
\path [draw = red, very thick,  shift = {(5.25,0)}](-1.5,-2.59808) -- (1.5,2.59808);
\path [draw = red,very thick, shift = {(6,-0.649519)}](-3,0) -- (3,0);
\path [draw = red, very thick,  shift = {(5.25,0)}](-1.5,2.59808) -- (1.5,-2.59808);
\path [draw = red, very thick,  shift = {(6,0)}](-1.5,-2.59808) -- (1.5,2.59808);
\path [draw = red,very thick, shift = {(6,0)}](-3,0) -- (3,0);
\path [draw = red, very thick,  shift = {(6,0)}](-1.5,2.59808) -- (1.5,-2.59808);
\path [draw = red, very thick,  shift = {(6.75,0)}](-1.5,-2.59808) -- (1.5,2.59808);
\path [draw = red,very thick, shift = {(6,0.649519)}](-3,0) -- (3,0);
\path [draw = red, very thick,  shift = {(6.75,0)}](-1.5,2.59808) -- (1.5,-2.59808);
\path [draw = red, very thick,  shift = {(7.5,0)}](-1.5,-2.59808) -- (1.5,2.59808);
\path [draw = red,very thick, shift = {(6,1.29904)}](-3,0) -- (3,0);
\path [draw = red, very thick,  shift = {(7.5,0)}](-1.5,2.59808) -- (1.5,-2.59808);
\fill[fill=black, fill opacity=1.](6,0.) circle (0.0675cm);\end{tikzpicture}
  \caption{\small 
$\m$-minimal alcoves in the $\m$-Shi arrangement for $\m=1$ ($m=2$).
Dominant alcoves are shaded yellow (and/or blue, respectively).
}
  \label{fig;onetwoshi}
\end{figure}
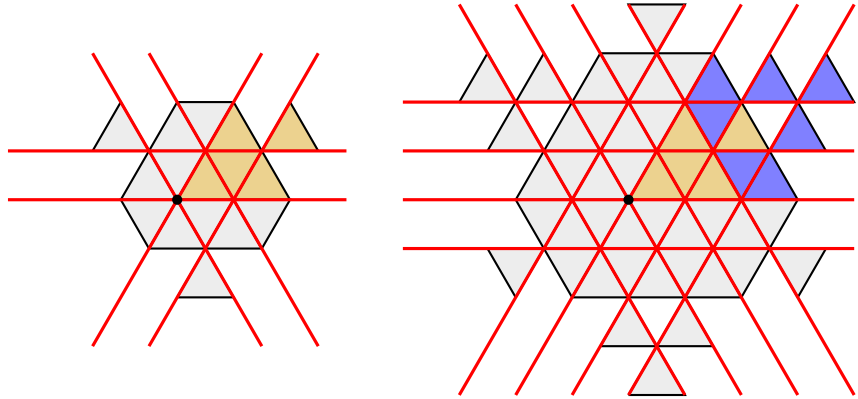

\begin{figure}[ht]

\begin{tikzpicture}[fill opacity=.5]
\tikzstyle{every node}=[font=\footnotesize]
[0,0,0]
\filldraw [draw=black,thick,fill opacity = .5, fill = Goldenrod,shift={(0.000000,0)}](0,0) -- (60:1cm) -- (0:1cm) --(0,0) node[blue] at (0.4cm,0.2cm) {} node[black] at (0.4cm,0.4cm) {} node[blue] at (0.4cm,0.6cm) {};
[1,0,0]
\filldraw [draw=black, thick,fill = Goldenrod, shift={(1,0)}](0,0) -- (60:1cm) -- (120:1cm) --(0,0) node[green,fill opacity=1] at (0cm,0.8cm) {} node[black,fill opacity=1, draw opacity = 1] at (0cm,0.6cm) {} node[green,fill opacity=1] at (0cm,0.4cm) {};
[1,0,1]
\filldraw [draw=black,thick,fill opacity = .5, fill = Goldenrod,shift={(1.000000,0)}](0,0) -- (60:1cm) -- (0:1cm) --(0,0) node[blue] at (0.4cm,0.2cm) {} node[black] at (0.4cm,0.4cm) {} node[blue] at (0.4cm,0.6cm) {};
[2,0,1]
\filldraw [draw=black, thick,fill = blue, shift={(2,0)}](0,0) -- (60:1cm) -- (120:1cm) --(0,0) node[green,fill opacity=1] at (0cm,0.8cm) {} node[black,fill opacity=1, draw opacity = 1] at (0cm,0.6cm) {} node[green,fill opacity=1] at (0cm,0.4cm) {};
[2,0,2]
\filldraw [draw=black,thick,fill opacity = .5, fill = blue,shift={(2.000000,0)}](0,0) -- (60:1cm) -- (0:1cm) --(0,0) node[blue] at (0.4cm,0.2cm) {} node[black] at (0.4cm,0.4cm) {} node[blue] at (0.4cm,0.6cm) {};
[1,1,0]
\filldraw [draw=black,thick,fill opacity = .5, fill = Goldenrod,shift={(0.500000,0.866025)}](0,0) -- (60:1cm) -- (0:1cm) --(0,0) node[blue] at (0.4cm,0.2cm) {} node[black] at (0.4cm,0.4cm) {} node[blue] at (0.4cm,0.6cm) {};
[2,1,0]
\filldraw [draw=black, thick,fill = blue, shift={(1.5,0.866025)}](0,0) -- (60:1cm) -- (120:1cm) --(0,0) node[green,fill opacity=1] at (0cm,0.8cm) {} node[black,fill opacity=1, draw opacity = 1] at (0cm,0.6cm) {} node[green,fill opacity=1] at (0cm,0.4cm) {};
[2,1,1]
\filldraw [draw=black,thick,fill opacity = .5, fill = Goldenrod,shift={(1.500000,0.866025)}](0,0) -- (60:1cm) -- (0:1cm) --(0,0) node[blue] at (0.4cm,0.2cm) {} node[black] at (0.4cm,0.4cm) {} node[blue] at (0.4cm,0.6cm) {};
[3,1,2]
\filldraw [draw=black,thick,fill opacity = .5, fill = blue,shift={(2.500000,0.866025)}](0,0) -- (60:1cm) -- (0:1cm) --(0,0) node[blue] at (0.4cm,0.2cm) {} node[black] at (0.4cm,0.4cm) {} node[blue] at (0.4cm,0.6cm) {};
[2,2,0]
\filldraw [draw=black,thick,fill opacity = .5, fill = blue,shift={(1.000000,1.73205)}](0,0) -- (60:1cm) -- (0:1cm) --(0,0) node[blue] at (0.4cm,0.2cm) {} node[black] at (0.4cm,0.4cm) {} node[blue] at (0.4cm,0.6cm) {};
[3,2,1]
\filldraw [draw=black,thick,fill opacity = .5, fill = blue,shift={(2.000000,1.73205)}](0,0) -- (60:1cm) -- (0:1cm) --(0,0) node[blue] at (0.4cm,0.2cm) {} node[black] at (0.4cm,0.4cm) {} node[blue] at (0.4cm,0.6cm) {};
[4,2,2]
\filldraw [draw=black,thick,fill opacity = .5, fill = blue,shift={(3.000000,1.73205)}](0,0) -- (60:1cm) -- (0:1cm) --(0,0) node[blue] at (0.4cm,0.2cm) {} node[black] at (0.4cm,0.4cm) {} node[blue] at (0.4cm,0.6cm) {};
\path [draw = red, very thick, draw opacity = 1, shift = {(0,0)}] (0,0) -- +(60:4);
\path [draw = red,very thick, draw opacity = 1] (0,0) -- +(4,0);
\path [draw = red,very thick,draw opacity = 1] (0,0) -- (0,0);
\path [draw = red, very thick, draw opacity = 1, shift = {(1,0)}] (0,0) -- +(60:4);
\path [draw = red,very thick, draw opacity = 1] (0.5,0.866025) -- +(4,0);
\path [draw = red,very thick,draw opacity = 1] (0.5,0.866025) -- (1,0);
\path [draw = red, very thick, draw opacity = 1, shift = {(2,0)}] (0,0) -- +(60:4);
\path [draw = red,very thick, draw opacity = 1] (1,1.73205) -- +(4,0);
\path [draw = red,very thick,draw opacity = 1] (1,1.73205) -- (2,0);
\begin{pgfonlayer}{foreground layer}
\node[black, shift={(0.592949,0.433013)}]{${\scriptscriptstyle e}$};
\end{pgfonlayer}{foreground layer}
\begin{pgfonlayer}{foreground layer}
\node[black, shift={(1.05,0.433013)}]{${\scriptscriptstyle 0}$};
\end{pgfonlayer}{foreground layer}
\begin{pgfonlayer}{foreground layer}
\node[black, shift={(1.54295,0.433013)}]{${\scriptscriptstyle 20}$};
\end{pgfonlayer}{foreground layer}
\begin{pgfonlayer}{foreground layer}
\node[black, shift={(2,0.433013)}]{${\scriptscriptstyle 120}$};
\end{pgfonlayer}{foreground layer}
\begin{pgfonlayer}{foreground layer}
\node[black, shift={(2.49295,0.433013)}]{${\scriptscriptstyle 0120}$};
\end{pgfonlayer}{foreground layer}
\begin{pgfonlayer}{foreground layer}
\node[black, shift={(1.04295,1.29904)}]{${\scriptscriptstyle 10}$};
\end{pgfonlayer}{foreground layer}
\begin{pgfonlayer}{foreground layer}
\node[black, shift={(1.5,1.29904)}]{${\scriptscriptstyle 210}$};
\end{pgfonlayer}{foreground layer}
\begin{pgfonlayer}{foreground layer}
\node[black, shift={(1.99295,1.29904)}]{${\scriptscriptstyle 1210}$};
\end{pgfonlayer}{foreground layer}
\begin{pgfonlayer}{foreground layer}
\node[black, shift={(2.94295,1.29904)}]{${\scriptscriptstyle 201210}$};
\end{pgfonlayer}{foreground layer}
\begin{pgfonlayer}{foreground layer}
\node[black, shift={(1.49295,2.16506)}]{${\scriptscriptstyle 0210}$};
\end{pgfonlayer}{foreground layer}
\begin{pgfonlayer}{foreground layer}
\node[black, shift={(2.44295,2.16506)}]{${\scriptscriptstyle 010210}$};
\end{pgfonlayer}{foreground layer}
\begin{pgfonlayer}{foreground layer}
\node[black, shift={(3.39295,2.16506)}]{${\scriptscriptstyle 12010210}$};
\end{pgfonlayer}{foreground layer}
\path [draw = lightgray, very thick,  shift = {(8,0)}](-1.5,-2.59808) -- (1.5,2.59808);
\path [draw = lightgray,very thick, shift = {(8,0)}](-3,0) -- (3,0);
\path [draw = lightgray, very thick,  shift = {(8,0)}](-1.5,2.59808) -- (1.5,-2.59808);
\path [draw = lightgray, very thick,  shift = {(9,0)}](-1.5,-2.59808) -- (1.5,2.59808);
\path [draw = lightgray,very thick, shift = {(8,0.866025)}](-3,0) -- (3,0);
\path [draw = lightgray, very thick,  shift = {(9,0)}](-1.5,2.59808) -- (1.5,-2.59808);
\path [draw = Goldenrod, very thick,  shift = {(8,0)}](0.5,2.59808) -- (2.5,-0.866025) -- (-1.5,-0.866025) -- cycle; 
\path [draw = blue, very thick,  shift = {(8,0)}](0.5,4.33013) -- (4,-1.73205) -- (-3,-1.73205) -- cycle; 
\fill[fill=Goldenrod, fill opacity=1.](8,0) circle (0.09cm) node [anchor = south] {$$};
\fill[fill=Goldenrod, fill opacity=1.](9.5,0.866025) circle (0.09cm) node [anchor = south] {$s_0$};
\fill[fill=Goldenrod, fill opacity=1.](9.5,-0.866025) circle (0.09cm) node [anchor = south] {$s_2s_0$};
\fill[fill=blue, fill opacity=1.](6.5,0.866025) circle (0.09cm)node [anchor = south] {$s_1s_2s_0$};
\fill[fill=blue, fill opacity=1.](9.5,2.59808) circle (0.09cm)node [anchor = south] {$s_0s_1s_2s_0$};
\fill[fill=Goldenrod, fill opacity=1.](8,1.73205) circle (0.09cm) node [anchor = south] {$s_1s_0$};
\fill[fill=blue, fill opacity=1.](8,-1.73205) circle (0.09cm)node [anchor = south] {$s_2s_1s_0$};
\fill[fill=Goldenrod, fill opacity=1.](6.5,-0.866025) circle (0.09cm) node [anchor = south] {$s_1s_2s_1s_0$};
\fill[fill=blue, fill opacity=1.](11,-1.73205) circle (0.09cm)node [anchor = south] {$s_2s_0s_1s_2s_1s_0$};
\fill[fill=blue, fill opacity=1.](11,0) circle (0.09cm)node [anchor = south] {$s_0s_2s_1s_0$};
\fill[fill=blue, fill opacity=1.](8,3.4641) circle (0.09cm)node [anchor = south] {$s_0s_1s_0s_2s_1s_0$};
\fill[fill=blue, fill opacity=1.](5,-1.73205) circle (0.09cm)node [anchor = south] {$s_1s_2s_0s_1s_0s_2s_1s_0$};
\draw[draw=black](8,0.) circle (0.11cm);
\end{tikzpicture}
  \caption{\small 
$w$ in the alcove  $w^{-1} \Afund$; and
the corresponding
$\beta = w (0, \ldots, 0)$
}
  \label{fig;rightleft}
\end{figure}


\section{
Narayana numbers
}

\label{sec-Narayana}
In this section we further  refine
the enumeration of $\n$-cores $\lambda$ which are also $(\m \n +1)$-cores.
We count by the number of residues $i$ such that $\lambda$ has
exactly $\m$  removable $i$-boxes. 
This refinement produces the $\m$-Narayana numbers, or
generalized Narayana numbers,  $N^m(k)$.
Hence we add another set to
Athanasiades'  list in Theorem 1.2 of \cite{A2005a} of combinatorial
objects counted by generalized Narayana numbers.
We recommend the reader see his Section 5.1 for an excellent discussion
of the $\m$-Shi arrangement in type $A$.

Recall
Equation \eqref{eq-wtlambda}
from Section \ref{sec-abacus} that
a partition has weight
${\mathrm wt}(\lambda) =
\Lambda_0 -  \sum_{(x,y) \in \lambda} \alpha_{(y-x) \, \mod n}.$

It is well-known that $s_i \lambda = \mu$ iff
$s_i {\mathrm wt}(\lambda) = {\mathrm wt} (\mu)$ 
where the action of $\affS$ on the weight lattice is given by
$$s_i (\gamma) = \gamma - \brac{\gamma}{ \alpha_i}\alpha_i.$$
We refer the reader to Chapters 5,6 of \cite{Kac}
for details on the affine weight lattice,
definition of $\Lambda_0$ and so on.
For computational purposes, all we need remind the reader of is that
$\brac{\Lambda_0}{\alpha_i} = \delta_{i,0}$ and
$\brac{\alpha_0}{\alpha_i}  = 2 \delta_{i,0} - \delta_{i,1} - \delta_{i,n-1}$.
\begin{remark}
In other words, Equation \eqref{eq-wtlambda} says that
if $s_i$ removes $k$ boxes (of residue $i$) from $\lambda$, or adds
$-k$ boxes to $\lambda$ to obtain $\mu$,
then ${\mathrm wt}(\mu) = 
s_i ({\mathrm wt}(\lambda)) = {\mathrm wt}(\lambda)  - k \alpha_i$.
\end{remark}

A straightforward rephrasing of Proposition \ref{prop;hyper} is then:

\begin{proposition}
\label{prop;narayana}
Let $\lambda$ be an $\n$-core, $k \in \Z_{>0}$, and $w \in \affS$ of 
minimal length such that $w \emptyset = \lambda$. 
Fix $0 \le i < n$.
The following are equivalent
\begin{enumerate}
\item
$s_i \lambda = \lambda \setminus\,  k $ boxes of residue $i$,
\item
$\brac{\vec \n (\lambda)}{\alpha_i} = -k$ for $i \neq 0$,
$\quad \brac{\vec \n (\lambda)}{\theta} = k+1$ for $i = 0$,
\item
$w^{-1} \Afund \subseteq \Hp{\alpha}{k},
\quad w^{-1} s_i \Afund \subseteq \Hn{\alpha}{k}$
where $w^{-1} (\alpha_i) = \alpha -k \delta$.
\end{enumerate}
\end{proposition}

When we rephrase Corollary \ref{cor;minimal} in this context, it says:

Suppose $\lambda = w \emptyset$ is the $\n$-core associated to the dominant
alcove $\A = w^{-1} \Afund$.
Then $\A$ is $\m$-minimal iff whenever $\lambda$ has exactly $k$ removable
boxes of residue $i$ then $k \le m$.  In this case, $\lambda$ is also
an $(\m\n + 1)$-core.

\begin{example}
\label{ex;hyp_inv_two}
Here we continue the example of Example \ref{ex;hyp_inv_one}.
Consider the $3$-core $\lambda = (5,3,2,2,1,1) = w \emptyset$ for
$w = s_1 s_2 s_0 s_1 s_2 s_1 s_0 
 =s_2 t_{(-2,0,2)}= t_{(-2,2,0)} s_2$.
Let $\A = w^{-1} \Afund$, which is pictured in Figure \ref{fig;hyper_inv},
and recall
$w^{-1}(\alpha_0) = {\color{Blue} -\alpha_1 + 3 \delta}$
corresponding to 
$\A \subseteq {\color{Blue} \Hp{\alpha_1}{2}}$ and 
$\A \subseteq {\color{Blue} \Hp{\alpha_1}{1}}$, but $\A \subseteq \Hn{\alpha_1}{3}$.
Also
$w^{-1}(\alpha_1) = {\color{Red}  \theta - 4 \delta}$
corresponding to
$\A \subseteq {\color{Red} \Hp{\theta}{k}}$ for
$1\le k \le 4$, and
$w^{-1}(\alpha_2) = {\color{LimeGreen} -\alpha_2+ 2 \delta}$
corresponds to
$\A \subseteq  {\color{LimeGreen} \Hp{\alpha_2}{1}}.$

This data  also 
corresponds to $\lambda$'s
$3$ addable ${\color{Blue} 0}$-boxes,
$4$ removable ${\color{Red} 1}$-boxes,
$2$ addable ${\color{LimeGreen} 2}$-boxes, 
as explained in Proposition \ref{prop;narayana}.

\begin{figure}[ht] 
\label{fig;lambda-one}
\begin{center}
\parbox[t]{2in}{
$\lambda = $
\tableau{\mbox{0} & \mbox{1} & \mbox{2} & \mbox{0} & {\color{Red} \mbox{1} }\\
         \mbox{2} & \mbox{0}& {\color{Red}  \mbox{1} } \\
         \mbox{1} & \mbox{2} \\
         \mbox{0} &{\color{Red}  \mbox{1} }\\
         \mbox{2} \\
         {\color{Red} \mbox{1} }\\
} }
\end{center}
\end{figure}
\end{example}
\begin{example}
\label{ex;hyp_inv_four}
Consider the $4$-core $\lambda = (5,2,2,1,1,1) = w \emptyset$ for
$w = s_3 s_0 s_1 s_2 s_3 s_2 s_1 s_0$.
Let $\A = w^{-1} \Afund$.
Note $w^{-1} = s_3 t_{(2,0,-2,0)}$ and $\vec n (\lambda) = (2,0,-2,0)$.
Then 
$$\A \subseteq \Hp{\alpha_3}{1} \cap \Hn{\alpha_2+\alpha_3}{2}
\cap \Hn{\alpha_1}{2} \cap \Hp{\alpha_1 +\alpha_2}{2},$$
which corresponds to $\lambda$'s $1$ removable $0$-box,
$2$ addable $1$-boxes, $2$ addable $2$-boxes, $2$ removable $3$-boxes. 
Note $w^{-1}(\alpha_0) = \alpha_3-\delta,$
$w^{-1}(\alpha_1) =-(\alpha_2+\alpha_3) + 2 \delta$,
$w^{-1}(\alpha_2) =-\alpha_1 + 2 \delta$,
$w^{-1}(\alpha_3) =(\alpha_1+\alpha_2) - 2 \delta$.

\begin{figure}[ht] 
\label{fig;lambda-two}
\begin{center}
\parbox[t]{2in}{
$\lambda = $
\tableau{\mbox{0} & \mbox{1} & \mbox{2} & \mbox{3} & \mbox{0} \\
         \mbox{3} & \mbox{0} \\
         \mbox{2} & \mbox{3} \\
         \mbox{1} \\
         \mbox{0} \\
         \mbox{3} \\
} }
\end{center}
\end{figure}
\end{example}

We could conversely construct $\lambda$ from
the number of removable/addable boxes of each residue,
since this data describes $\wt(\lambda)$ and hence determines $\lambda$.
However, in practice executing the bijection of Section \ref{sec-bij},
one may as well find $w$, for instance, as follows.
Given an $\n$-core $\lambda$, determine $\vec n (\lambda)$ as described
in Section \ref{sec-abacus}.  Let $u \in \Sn$ be of minimal length such that
$u(-\vec n (\lambda))$ is dominant.  Then $w = u^{-1} t_{\vec n (\lambda)}$
and $w^{-1} \Afund = t_{-\vec n (\lambda)} u \Afund
\subseteq
\Hn{u^{-1}(\theta)}{1- \brac{\vec n (\lambda)}{\theta}} \cap
\bigcap_{i=1}^{n-1} \Hp{u^{-1}(\alpha_i)}{-\brac{\vec n (\lambda)}{\alpha_i}}$
which describes where $w^{-1} \Afund$ is located in $V$.
This process gives a slightly more constructive version of the bijection
in Theorem \ref{thm;main}.

We also note that one can easily read off the coordinates $k(w^{-1}, \alpha)$
that Shi uses in \cite{Shi} from the data above,
describing where $w^{-1} \Afund$ is located in $V$ with respect to the (affine)
root hyperplanes.

\subsection{A refinement}
Proposition \ref{prop;narayana} thus gives us
another combinatorial way to count the $\m$-Narayana numbers, as in
\cite{A2005a}. 
Recall the $k^{th}$  {\em $\m$-Narayana number}
 $N^m(k)$ counts how many dominant regions
of the $\m$-Shi arrangement
have exactly $k$ hyperplanes $\Hak{\alpha}{m}$
separating them from $\Afund$ such that $\Hak{\alpha}{m}$ contains a wall
of the region.

In other words, for fixed $k$, we count how many $\m$-minimal alcoves $\A 
= w^{-1} \Afund$
satisfy that for exactly $k$ positive roots $\alpha$, there exists an $i$ 
such that  $w^{-1} \Afund  \subseteq \Hp{\alpha}{m}$
but $w^{-1} s_i \Afund  \subseteq \Hn{\alpha}{m}$.
It is clear that
$$\sum_{k\ge 0} N^m(k) = \m \text{-Catalan number} 
$$
since each dominant $\m$-minimal alcove gets counted once.

By Proposition \ref{prop;narayana} above, $N^m(k)$ equivalently counts
how many $\n$-cores $\lambda$ that are also $(\m\n + 1)$-cores
 have exactly $k$ distinct residues $i$ such that $\lambda$
has precisely $\m$  removable $i$-boxes.
See Example \ref{ex;hyp_inv_two} below.

\begin{corollary}
\label{cor;narayana}
Let $N^m(k)$ denote the $\m$-Narayana number of type $A_{n-1}$.
Then 
\begin{align*}
N^m(k)  =  
|  \{
& \lambda \mid \text{$\lambda$ is an $\n$-core  and
$(\m\n + 1)$-core and $\exists K \subseteq \Z/\n \Z$} \\
& \text{
with $|K| = k$ such that $\lambda$ has exactly $\m$ removable boxes} \\
&\text{
of residue $i$ iff $i \in K$} \} |.
\end{align*}
\end{corollary}

\begin{example}
\label{ex;narayana}
For $\n=3$, $\m=2$, the $\m$-Catalan number is $12 = 5+6+1$. 
\begin{align*}
N^2(0) =& 5 =
| \{ \emptyset, (1), (2), (1,1), (3,1,1) \} |
\\
N^2(1) =& 6 =
| \{  (3,1), (2,1,1), (2,2,1,1), (4,2), (5,3,1,1), (4,2,2,1,1) \} |
\\
N^2(2) =& 1 =
| \{ 
(6,4,2,2,1,1) \} |
\end{align*}
\end{example}

\bibliographystyle{plain}

\end{document}